\documentclass[12pt,reqno]{amsart}
\usepackage{amsmath}
\usepackage{amssymb}
\usepackage{amstext}
\usepackage{bm}
\usepackage{physics}
\usepackage{a4wide}
\usepackage{txfonts}
\usepackage{graphicx}
\usepackage[numbers,sort&compress]{natbib}
\allowdisplaybreaks \numberwithin{equation}{section}
\usepackage{color}
\usepackage{cases}
\usepackage{comment}
\usepackage{enumitem}
\usepackage{booktabs}
\usepackage{tabularx}
\usepackage{array}
\usepackage{caption}
\usepackage{subcaption}
\graphicspath{{images/}}
\numberwithin{equation}{section}

\newtheorem{theorem}{Theorem}[section]
\newtheorem{proposition}[theorem]{Proposition}
\newtheorem{corollary}[theorem]{Corollary}
\newtheorem{lemma}[theorem]{Lemma}
\newtheorem*{theoremA}{Theorem A}

\theoremstyle{definition}

\theoremstyle{remark}
\newtheorem{remark}[theorem]{Remark}

\begin{document}

	\title
	[Quantitative Stability of First Laplacian Eigenstates]{Quantitative Stability of First Laplacian Eigenstates for the Incompressible Euler Equation on a Flat 2-Torus}

	\author{Fatao Wang}
	\address{School of Mathematical Sciences, Dalian University of Technology, Dalian 116024, PR China}
	\email{wangfatao@mail.dlut.edu.cn}
	
	\author{Guodong Wang}
	\address{School of Mathematical Sciences, Dalian University of Technology, Dalian 116024, PR China}
	\email{gdw@dlut.edu.cn}

	\author{Xiaoxue Zhu}
	\address{School of Mathematical Sciences, Dalian University of Technology, Dalian 116024, PR China}
	\email{zxx12501032@mail.dlut.edu.cn}

	
	\begin{abstract}
In this paper, we establish quantitative estimates for the orbital stability of the first Laplacian eigenstates of the incompressible Euler equation on a two-dimensional flat torus. We focus mainly on the hexagonal torus, where the first Laplacian eigenspace has a more intricate structure and the Casimir functionals may exhibit strong degeneracy at special amplitude and phase configurations.
The main novelty of the proof is to reduce the estimates for the amplitude parameters of the perturbed solution to a root-stability problem for a cubic polynomial under coefficient perturbations, thereby overcoming the strong degeneracy in an effective way.
These estimates appear to indicate that stronger degeneracy in the amplitude-phase configuration leads to weaker stability.
\end{abstract}

	\maketitle
	\tableofcontents
	\section{Introduction}\label{sec1}
	\subsection{Motivation}
	
	Let $\Lambda$ be a two-dimensional lattice; that is, there exist two linearly independent vectors $\bm{\xi}, \bm{\eta}\in\mathbb R^2$ such that
	\[
	\Lambda = \left\{ m \bm{\xi} + n \bm{\eta} \mid m, n \in \mathbb{Z} \right\}.
	\]
	The pair $(\bm{\xi}, \bm{\eta})$ is called a basis of $\Lambda$.
	Consider the flat 2-torus $\mathbb T_\Lambda:=\mathbb R^2/\Lambda$ and the incompressible Euler equation on $\mathbb T_\Lambda$ in vorticity form:
	\begin{equation}\label{euler}
		\left\{
		\begin{aligned}
			&\partial_t \omega + \nabla^\perp\psi \cdot \nabla \omega = 0,
			\qquad t \in \mathbb R,\ \mathbf x = (x_1,x_2) \in \mathbb T_\Lambda,\\
			&\psi=  (-\Delta)^{-1}\omega,
		\end{aligned}
		\right.
	\end{equation}
	where $\omega$ is the scalar vorticity, $\psi$ is the stream function, and $\nabla^\perp:= (\partial_{x_2}, -\partial_{x_1})$. Note that $\omega$ is assumed to have mean zero to ensure that $\psi$ is well defined. The global well-posedness theory for \eqref{euler} has been well developed in various function spaces \cite{Be84,De9125,Di87,Yu63}. In particular, global well-posedness holds in the class of mean-zero $C^1$ functions.
	However, a comprehensive theory describing the evolution of solutions remains far from complete \cite[Sect.~3]{Drivas23}.
	Among the many problems concerning evolution, a fundamental one is to study the long-time behavior of solutions near various steady states.

	In this paper, we study the Lyapunov stability of a class of steady states, referred to as \emph{Laplacian eigenstates}, whose stream functions are given by eigenfunctions of the following Laplacian eigenvalue problem:
	\begin{equation}\label{lep}
		\begin{cases}
			-\Delta u = \lambda u, & \mathbf{x} \in \mathbb T_\Lambda,\\
			\displaystyle \int_{\mathbb T_\Lambda} u \, \dd{\mathbf{x}} = 0.
		\end{cases}
	\end{equation}
	Note that all eigenfunctions of \eqref{lep} are sinusoidal and can be obtained by computing the dual lattice of $\Lambda$. For the first eigenspace, denoted by $\mathbf E_1$, the following properties hold  (see \cite[Sect.~2]{W252}):
	\begin{itemize}
		\item[(1)] The dimension of $\mathbf E_1$ is either $2$, $4$, or $6$.
		
		\item[(2)] $\dim(\mathbf{E}_1) = 6$ if and only if $\mathbb T_\Lambda$ is a \emph{hexagonal torus}; that is,  $\Lambda$ admits a basis $(\bm{\xi}, \bm{\eta})$ such that
		\[
		|\bm{\xi}| = |\bm{\eta}|, \quad \angle(\bm{\xi}, \bm{\eta}) = \pi/3.
		\]
		
		\item[(3)] If $\mathbb T_\Lambda$ is a hexagonal torus, then there exist three nonzero vectors $\mathbf{k}_1, \mathbf{k}_2, \mathbf{k}_3$ satisfying
		\[
		|\mathbf{k}_1| = |\mathbf{k}_2| = |\mathbf{k}_3|,
		\quad
		\mathbf{k}_3 = \mathbf{k}_1 + \mathbf{k}_2,
		\]
		such that
		\begin{equation}\label{e1df}
		\mathbf{E}_1 = \operatorname{span}\left\{
		\cos(\mathbf{k}_1 \cdot \mathbf{x}),\, \sin(\mathbf{k}_1 \cdot \mathbf{x}),\,
		\cos(\mathbf{k}_2 \cdot \mathbf{x}),\, \sin(\mathbf{k}_2 \cdot \mathbf{x}),\,
		\cos(\mathbf{k}_3 \cdot \mathbf{x}),\, \sin(\mathbf{k}_3 \cdot \mathbf{x})
		\right\}.
		\end{equation}
	\end{itemize}

	For the square torus, orbital stability (up to translations) of the first Laplacian eigenstates has been established in \cite{E24,WZ23,WS99}, and has recently been extended to flat $2$-tori of \emph{arbitrary shape} in \cite{W252}.
	
	\begin{theoremA}[\cite{W252}]
		Fix $1<p<\infty$ and $\bar\omega \in \mathbf E_1$. Then, for any $\varepsilon>0$, there exists $\delta>0$, depending on $\bar\omega$ and $\varepsilon$, such that for any mean-zero $C^1$ solution $\omega$ to \eqref{euler},
		\[
		\|\omega(0,\cdot)-\bar\omega\|_{L^p(\mathbb T_\Lambda)}<\delta
		\quad\Longrightarrow\quad
		{\rm dist}_p(\omega(t,\cdot),\mathbf O_{\bar\omega}) <\varepsilon \quad \forall\,t\in\mathbb R,
		\]
		where $\mathbf O_{\bar\omega}$ denotes the translational orbit of $\bar\omega$,
		\begin{equation}\label{defoforbit}
\mathbf O_{\bar\omega}:=\left\{\bar\omega(\cdot-\mathbf p)\mid \mathbf p\in\mathbb T_{\Lambda}\right\},
\end{equation}
		and ${\rm dist}_p$ denotes the $L^p$ distance.
	\end{theoremA}

	However, the dependence of $\delta$ on $\varepsilon$ remains unknown, except in the special case of the square torus \cite{E24}. Our goal in this paper is to provide explicit estimates for this dependence on tori of arbitrary shape.

	\subsection{Main theorem}
	We present only the quantitative estimates for the most delicate hexagonal case as the main theorem below, and postpone the simpler non-hexagonal cases to the final section.
	
	\begin{theorem}\label{thm6d}
		Let $\mathbb T_\Lambda$ be a hexagonal torus, and let $\mathbf E_1$ be given by \eqref{e1df}.
 Fix $\bar{\omega} \in \mathbf E_1$ of the form
		\begin{equation*}
			\bar \omega(\mathbf x)=\sum_{i=1}^3A_i\cos(\mathbf k_i\cdot \mathbf x+\alpha_i),\quad A_i\geq 0,\,\, \alpha_i\in\mathbb R,\,\, (A_1,A_2,A_3)\neq (0,0,0).
		\end{equation*}
	Then there exist  $\varepsilon_0>0$ and $C>0$, depending only on $\Lambda$ and $\bar\omega$, such that for any $\varepsilon\in(0,\varepsilon_0)$ and any mean-zero $C^1$ solution $\omega(t,\mathbf{x})$ to \eqref{euler}, it holds that
		\[
		{\rm dist}_{2}(\omega(0,\cdot), \mathbf O_{\bar\omega}) < \varepsilon
		\quad\Longrightarrow\quad
		{\rm dist}_{2}(\omega(t,\cdot), \mathbf O_{\bar\omega}) < C \varepsilon^\gamma
		\quad\text{for all } t \in \mathbb{R},
		\]
		where $\mathbf O_{\bar\omega}$ is given by \eqref{defoforbit}, and $\gamma$ is determined by $A_i$ and $\alpha:=\alpha_1+\alpha_2-\alpha_3$ as follows:

		\begin{itemize}
			\item[(1)]  $A_1A_2A_3=0$:
			\begin{itemize}
				\item[(1-1)] If exactly one of $A_1,A_2,A_3$ is zero, then $\gamma=\frac12$.
				\item[(1-2)] If exactly two of $A_1,A_2,A_3$ are zero, then $\gamma=\frac14$.
			\end{itemize}
			
			\item[(2)] $A_1A_2A_3\neq 0$:
			\begin{itemize}
				\item[(2-1)] If $A_1,A_2,A_3$ are pairwise distinct and $\sin\alpha\neq 0$, then $\gamma=1$.
				\item[(2-2)] If $A_1,A_2,A_3$ are pairwise distinct and $\sin\alpha=0$, then $\gamma=\frac12$.
				\item[(2-3)] If exactly two of $A_1,A_2,A_3$ are equal and $\sin\alpha\neq 0$, then $\gamma=\frac12$.
				\item[(2-4)] If exactly two of $A_1,A_2,A_3$ are equal and $\sin\alpha=0$, then $\gamma=\frac14$.
				\item[(2-5)] If $A_1=A_2=A_3$ and $\sin\alpha\neq 0$, then $\gamma=\frac13$.
				\item[(2-6)] If $A_1=A_2=A_3$ and $\sin\alpha=0$, then $\gamma=\frac16$.
			\end{itemize}
		\end{itemize}
		
	\end{theorem}

	We give some remarks on Theorem \ref{thm6d} below.
	
\begin{remark}
Theorem \ref{thm6d} holds for more general perturbations, not only for $C^1$ solutions. Indeed, our proof only requires that $\omega \in C(\mathbb R; \mathring L^2(\mathbb T_\Lambda))$ (see \eqref{defoflo} for the definition of $\mathring L^2(\mathbb T_\Lambda)$) and that both the kinetic energy $\mathsf E$ and various Casimir functionals $\mathsf C_F$ (see \eqref{defofe} and \eqref{defofc} for the definitions) are conserved.
\end{remark}

\begin{remark}
It remains unclear whether the stability exponent $\gamma$ is optimal, except in case (2-1). This is an interesting but difficult problem for future investigation. Similar issues also arise in \cite{DM25,E24,WW25}.
\end{remark}

	\begin{remark}
		
		In Remark \ref{symmetry-degeneracy}, we will discuss in detail the dependence of the constant $C$ on the amplitude and phase configuration $(A_1, A_2, A_3,\alpha)$, from which one can see how the exponent $\gamma$ exhibits a change in order as the amplitude and phase configuration varies.
	\end{remark}

\begin{remark}
The table below shows more clearly how the value of $\gamma$ depends on $(A_1,A_2,A_3,\alpha)$. It follows from the table that stability becomes weaker as the amplitudes become more degenerate, either through vanishing components or through coincidences among $A_1,A_2,A_3$. When all amplitudes are nonzero, the additional phase degeneracy $\sin\alpha=0$ further reduces the stability exponent $\gamma$ by a factor of two. However, this interpretation is only heuristic, since the sharpness of the exponent $\gamma$ remains unclear. A similar weakening of stability for certain special flow patterns  also appears in \cite{DM25,E24,WW25}.

	\medskip
	
	\begin{center}
		\scriptsize
		\renewcommand{\arraystretch}{1.18}
		\setlength{\tabcolsep}{5pt}
		\captionsetup{skip=2pt}
		\begin{minipage}[t]{0.42\textwidth}
			\centering
			\captionof*{table}{$A_1A_2A_3=0$}
			
			\begin{tabular}{lc}
				\toprule
				\textbf{Amplitude configuration}
				& \textbf{Any $\alpha$} \\
				\midrule
				exactly one of $A_1,A_2,A_3$ is zero
				& $\frac12$ \\
				exactly two of $A_1,A_2,A_3$ are zero
				& $\frac14$ \\
				\bottomrule
			\end{tabular}
		\end{minipage}
		\hfill
		\begin{minipage}[t]{0.54\textwidth}
			\centering
			\captionof*{table}{$A_1A_2A_3\neq0$}
			
			\begin{tabular}{lcc}
				\toprule
				\textbf{Amplitude configuration}
				& \textbf{$\sin\alpha\neq0$}
				& \textbf{$\sin\alpha=0$} \\
				\midrule
				$A_1,A_2,A_3$ are pairwise distinct
				& $1$
				& $\frac12$ \\
				exactly two of $A_1,A_2,A_3$ are equal
				& $\frac12$
				& $\frac14$ \\
				$A_1=A_2=A_3$
				& $\frac13$
				& $\frac16$ \\
				\bottomrule
			\end{tabular}
		\end{minipage}
		
	\end{center}

\end{remark}
	\medskip



	\subsection{Comments, proof strategy, and novelties}

	The Euler equation possesses a rich family of conservation laws, which provides a natural framework for stability analysis. In the case of a flat 2-torus, the following quantities are conserved:
	\begin{itemize}
		\item [(1)] The kinetic energy $\mathsf E$, which in terms of vorticity can be expressed as
		\begin{equation}\label{defofe}\mathsf E(\omega)=\frac{1}{2}\int_{\mathbb T_\Lambda}\omega\,(-\Delta)^{-1}\omega\dd{\mathbf x}\footnote{Although $(-\Delta)^{-1}\omega$ may differ by an arbitrary constant, this integral is well-defined since $\omega$ has mean zero.};
\end{equation}
		\item  [(2)] The Casimir functionals of the form
		\begin{equation}\label{defofc}
\mathsf C_F(\omega)=\int_{\mathbb T_\Lambda}F(\omega)\dd{\mathbf x},
\end{equation}
		where $F\in C(\mathbb R)$.  In particular, the $k$-th order Casimir
		\begin{equation}\label{kcm}
			\mathsf C_k(\omega):=\int_{\mathbb T_\Lambda} \omega^k\dd{\mathbf x}
		\end{equation}
		is conserved, where $k$ is a positive integer.
	\end{itemize}

	In the 1960s, inspired by the Lyapunov function method in ODE theory, Arnold \cite{A65,A69} introduced the energy-Casimir (EC) functional and established two types of stability theorems, now known as Arnold's first and second stability theorems.
	For Laplacian eigenstates, the associated Casimir is the enstrophy, i.e., the $L^2$ norm of the vorticity. Using the EC functional method, it can be proved that for any $\bar\omega \in \mathbf E_1$, the first energy shell
	\[
	\mathbf S_{\bar\omega} := \left\{ v \in \mathbf E_1 \;\middle|\; \|v\|_{L^2(\mathbb T_{\Lambda})} = \|\bar\omega\|_{L^2(\mathbb T_{\Lambda})} \right\}
	\]
	is stable with respect to the enstrophy norm.
	However, this result cannot be improved using only the kinetic energy and the enstrophy, since these quantities do not distinguish between different states in $\mathbf S_{\bar\omega}$.
	The first refinement was obtained by Wirosoetisno and Shepherd \cite{WS99} for the square torus. They introduced higher-order Casimirs $\mathsf C_k,$ $k=2,4,6$, to distinguish elements of $\mathbf{S}_{\bar\omega}$, and studied the stability of the translational orbit $\mathbf O_{\bar\omega}$.
	Nevertheless, their formulation of stability relies on higher-order Casimirs, and  complete orbital stability  remains unresolved.
	A complete proof of the stability of $\mathbf O_{\bar\omega}$ for both rectangular and square tori was given in \cite{WZ23}, where they used a compactness argument originating from Burton's work \cite{BMA,B05}, reducing the problem to the equimeasurable partition of the first eigenspace. The stability result in \cite{WZ23} is qualitative, and is measured by the   $L^p$ norm of the vorticity for any $p\in(1,\infty)$.
	Later, Elgindi \cite{E24}  established a quantitative estimate on the stability of $\mathbf O_{\bar\omega}$ for the square torus in the $L^2$ setting, building on and refining the approach of Wirosoetisno and Shepherd. Similar ideas have since been applied to the orbital stability of Laplacian eigenstates on the sphere \cite{DM25} or in a disk \cite{WW25}. Very recently, the result of \cite{WZ23} was extended to tori of arbitrary shape in \cite{W252}, with the main contribution being the treatment of the hexagonal case.

The proof of Theorem \ref{thm6d} follows the framework introduced in \cite{E24,WS99} (see also \cite{DM25,WW25} for further developments) and is divided into two main steps:

	\begin{itemize}
		\item[(1)] We first establish the quantitative stability  of the first eigenspace $\mathbf E_1$. This step mainly relies on the conservation of the kinetic energy and the enstrophy, in combination with proper use of a Poincar\'e-type inequality.
		\item [(2)]	We then estimate the distance between the projection of the perturbed solution onto $\mathbf E_1$ and the translational orbit $\mathbf O_{\bar\omega}$. More specifically, we introduce modified integer-order Casimirs that are globally Lipschitz continuous in $L^2$ and use them to control the variation of the amplitude and phase parameters. The problem is then reduced to the analysis of a possibly degenerate nonlinear map in finite dimensions.
\end{itemize}

The first step is standard and follows from an argument similar to that in \cite{WW25}.
The main challenges and differences lie in the second step. First, the structure of the first eigenspace and the characterization of translational orbits within the first eigenspace are more complicated, which makes the analysis of energy and vorticity transfer within it more difficult. Second, the nonlinear map arising in the second step may be highly degenerate at certain special configurations, leading to weaker stability. For example, the stability exponent $\gamma$ may be as small as $1/6$ when $A_1=A_2=A_3$ and $\sin\alpha=0$, whereas the smallest value of stability exponent appearing in \cite{DM25,E24,WW25} was $1/2$. This strong degeneracy prevents us from directly applying the previous methods, such as the degenerate inverse function lemma established in \cite{DM25} and the direct computations carried out in \cite{WW25}. To overcome these difficulties, we make a key observation: \emph{the complicated nonlinear map in the second step can be suitably combined into a polynomial map given by the three elementary symmetric polynomials in three variables.} This reduces the estimates for the amplitude parameters to a \emph{root-stability problem} for a cubic polynomial under perturbations of its coefficients, which can be handled in a straightforward manner; see Lemma \ref{abn}. Once the estimates for the amplitude parameters have been established, the estimates for the phase parameters can be obtained through careful and intricate computations.

\subsection{Organization of the paper}
	
	The remainder of the paper is organized as follows. In Section \ref{sec2}, we establish a quantitative stability estimate for the first eigenspace $\mathbf E_1$.
	In Section \ref{sec3}, we prove a technical lemma on the quantitative stability of the roots of a polynomial under perturbations of its coefficients. This lemma plays a crucial role in the proof of the main theorem.
	Section \ref{sec4} is devoted to the proof of Theorem \ref{thm6d}.
Finally, in Section \ref{sec5}, we treat the remaining non-hexagonal cases and provide the corresponding quantitative stability results with brief proofs.

	\section{Quantitative stability of $\mathbf E_1$}\label{sec2}

	Denote by $\mathring L^2(\mathbb T_\Lambda)$ the closed subspace of mean-zero functions in  $L^2(\mathbb T_\Lambda)$, i.e.,
	\begin{equation}\label{defoflo}
\mathring L^2(\mathbb T_\Lambda):=\left\{v\in L^2(\mathbb T_\Lambda) \;\middle|\;  \int_{\mathbb T_\Lambda} v \dd{\mathbf x}=0\right\},
\end{equation}
	and denote by $\mathbf E_1^\perp$  the orthogonal complement of $\mathbf E_1$ in $\mathring L^2(\mathbb T_\Lambda)$.
	
	\begin{proposition}\label{prop21}
		Let $\omega_t := \omega(t,\cdot)$ be a mean-zero $C^1$ solution of the Euler equation \eqref{euler}. 	Consider the orthogonal decomposition
		\[\omega_t=v_t+w_t,\quad v_t \in \mathbf E_1,\quad w_t \in \mathbf E_1^{\perp}.\]
		Then, for all $t\in\mathbb R$, we have
		\[
		\|w_t\|_{L^2(\mathbb T_\Lambda)}
		\leq
		\sqrt{\frac{\lambda_2}{\lambda_2-\lambda_1}}\,
		\|w_0\|_{L^2(\mathbb T_\Lambda)},
		\]
		where $\lambda_i$ is the $i$-th eigenvalue (counted without multiplicity)  of \eqref{lep}.
	\end{proposition}
	\begin{remark}
		It is clear that
		\[
		{\rm dist}_2\bigl(\omega_t,\mathbf E_1\bigr) =\|w_t\|_{L^2(\mathbb T_\Lambda)}.
		\]
		Therefore, Proposition \ref{prop21} establishes a quantitative stability estimate for  $\mathbf E_1$.
		
	\end{remark}	
	
	To prove Proposition \ref{prop21}, we need the following Poincar\'e-type inequality associated with $\mathbf E_1$.
	
	\begin{lemma}\label{lema21}
		For any $v\in \mathbf E_1^\perp$, it holds that
		\begin{equation}\label{l2}
			\displaystyle \int_{\mathbb T_\Lambda} v\,(-\Delta)^{-1}v \dd{\mathbf{x}} \leq \frac{1}{\lambda_2}\displaystyle \int_{\mathbb T_\Lambda} v^2 \dd{\mathbf{x}}.
		\end{equation}

	\end{lemma}
	
	The proof of the above lemma is standard and follows by repeating the argument in \cite[Lemma 2.1]{WW25}.

	\begin{proof}[Proof of Proposition \ref{prop21}]
		
		Define the energy-Casimir functional $EC: \mathring L^2(\mathbb T_{\Lambda})\to \mathbb R$ by
		\[
		EC(v):=\int_{\mathbb T_\Lambda}v\,(-\Delta)^{-1}v \dd{\mathbf x}-\frac{1}{\lambda_1}\int_{\mathbb T_\Lambda} v^2\dd {\mathbf x},
		\]
		which is  conserved for any mean-zero $C^1$ solution of \eqref{euler}. A straightforward computation yields
		\begin{equation}\label{461}
			EC(\omega_t)  =\displaystyle \int_{\mathbb T_\Lambda} w_t\,(-\Delta)^{-1}w_t   \dd{\mathbf{x}}-\frac{1}{\lambda_1} \displaystyle \int_{\mathbb T_\Lambda} w_t^2 \dd{\mathbf x}.
		\end{equation}
		On the other hand, by Lemma \ref{lema21}, we have
		\begin{equation}\label{462}
			\int_{\mathbb T_\Lambda} w_t\,(-\Delta)^{-1}w_t \, \dd{\mathbf{x}} \leq \frac{1}{\lambda_2} \displaystyle \int_{\mathbb T_\Lambda} w_t^2 \, \dd{\mathbf{x}}.
		\end{equation}
		From \eqref{461} and \eqref{462}, we have
		\[
		-\frac{1}{\lambda_1} \displaystyle \int_{\mathbb T_\Lambda} w_t^2 \dd{\mathbf x}\leq EC(\omega_t) \leq \left(\frac{1}{\lambda_2}-\frac{1}{\lambda_1}\right)\int_{\mathbb T_\Lambda} {w_t}^2 \dd{\mathbf x}.
		\]
		Since $EC$ is conserved,  it follows that
		\[
		-\frac{1}{\lambda_1}\displaystyle \int_{\mathbb T_\Lambda} {w_0}^2 \dd{\mathbf x}\leq EC(\omega_0)=EC(\omega_t) \leq \left(\frac{1}{\lambda_2}-\frac{1}{\lambda_1}\right)\displaystyle  \int_{\mathbb T_\Lambda} {w_t}^2 \dd{\mathbf x},
		\]
		which implies that
		\[
		\displaystyle \int_{\mathbb T_\Lambda} {w_t}^2 \dd{\mathbf x} \leq \frac{\lambda_2}{\lambda_2-\lambda_1} \displaystyle \int_{\mathbb T_\Lambda} {w_0}^2 \dd{\mathbf x}.
		\]
		The proof is complete.
	\end{proof}

	\section{A technical lemma}\label{sec3}
	
	Let $n \ge 2$ be an integer. Denote by $f_1, \dots, f_n$ the  $n$ elementary symmetric polynomials in $n$ variables, i.e.,
	\[
	f_1(\mathbf a)=\sum_{i=1}^n a_i,\quad
	f_2(\mathbf a)=\sum_{1\le i<j\le n} a_i a_j,\quad
	\dots,\quad
	f_n(\mathbf a)=\prod_{i=1}^n a_i,
	\]
	where  $\mathbf a = (a_1,a_2, \dots, a_n) \in \mathbb{R}^n.$
	
	The following lemma gives quantitative bounds on \(|b_k-a_k|\) in terms of \(\sum_{i=1}^n |f_i(\mathbf a) - f_i(\mathbf b)|\), and plays a key role in proving the main theorem.

	\begin{lemma}\label{abn}
		Let $\mathbf{a} = (a_1, a_2, \dots, a_n) \in \mathbb{R}^n$, and fix $k \in \{1,2,\dots,n\}$. Suppose that $a_k$ appears $m$ times in the multiset $\{a_1, a_2, \dots, a_n\}$, i.e.,
		\[\# \left\{ i \in \{1,\dots,n\}  \;\middle|\; a_i = a_k \right\}=m.\]
		Then there exist $\delta=\delta(\mathbf a)>0$,  and $C=C(|\mathbf a|)>0$,  such that for any $\mathbf b\in\mathbb R^n$ satisfying $|b_k- a_k|<\delta,$
		\[
		\left|b_k-a_k\right|^m
		\le
		\frac{C}{\rho}\sum_{i=1}^n \left|f_i(\mathbf a)-f_i(\mathbf b)\right|,
		\qquad
		\rho:=\begin{cases}
			\prod_{\substack{1\le i\le n\\ a_i\ne a_k}} \left|a_i-a_k\right|,&\mbox{if }1\leq m<n,\\
			1,&\mbox{if } m=n.
		\end{cases}
		\]
	\end{lemma}
	\begin{proof}
		
		Define
		\[
		P_{\mathbf a}(s)
		:=
		s^n - f_1(\mathbf a)s^{n-1} + f_2(\mathbf a)s^{n-2} - f_3(\mathbf a)s^{n-3}
		+\cdots+(-1)^{n-1}f_{n-1}(\mathbf a)s + (-1)^n f_n(\mathbf a).
		\]
		Then, by Vieta's formula,
		\[
		P_{\mathbf a}(s)=\prod_{i=1}^n (s-a_i).
		\]
		Since $a_k$ appears $m$ times in the multiset $\{a_1, a_2, \dots, a_n\}$, $a_k$ is a root of multiplicity $m$ of $P_{\mathbf a}$. So we can write
		\begin{equation}\label{decom1}
			P_{\mathbf a}(s)=(s-a_k)^m Q_{\mathbf a}(s),
		\end{equation}
		where
		\[Q_{\mathbf a}(s):=
		\begin{cases}
			\prod_{\substack{1\le i\le n\\ a_i\ne a_k}}(s-a_i),&\mbox{if }1\leq m<n,\\
			1,&\mbox{if } m=n.
		\end{cases}
		\]
		It is clear that $|Q_{\mathbf a}(a_k)| =\rho$ according to the definition of $\rho$. Moreover, since $Q_{\mathbf a}$ is continuous, there exists some $\delta>0,$ depending only on $\mathbf a$, such that
		\begin{equation}\label{qb01}
			\left|Q_{\mathbf a}(b_k)\right|\geq \frac{\rho}{2}
		\end{equation}
		as long as $|a_k-b_k|<\delta.$ Without loss of generality, we assume that $\delta<1.$
		Define $P_{\mathbf b}$ analogously to $P_{\mathbf a}$,
		\[P_{\mathbf b}(s)
		:=
		s^n - f_1(\mathbf b)s^{n-1} + f_2(\mathbf b)s^{n-2} - f_3(\mathbf b)s^{n-3}
		+\cdots+(-1)^{n-1}f_{n-1}(\mathbf b)s + (-1)^n f_n(\mathbf b).\]
		Then
		\[
		P_{\mathbf a}(s)-P_{\mathbf b}(s)
		=
		\sum_{i=1}^n (-1)^i\bigl(f_i(\mathbf a)-f_i(\mathbf b)\bigr)s^{n-i}.
		\]
		Taking $s=b_k$, we get
		\[
		\left|P_{\mathbf a}(b_k)\right|\leq    \left( |\mathbf a|+1\right)^n  \sum_{i=1}^n \left|f_i(\mathbf a)-f_i(\mathbf b)\right|,
		\]
		where we used the fact that $P_{\mathbf b}(b_k)=0$, together with $|b_k|\leq   |a_k|+\delta\leq |\mathbf a|+1.$
		On the other hand, in view of \eqref{decom1},
		\[
		P_{\mathbf a}(b_k)=(b_k-a_k)^m Q_{\mathbf a}(b_k).
		\]
		So we further get
		\begin{equation}\label{qb02}
			\left|b_k-a_k\right|^m |Q_{\mathbf a}(b_k)|\le  \left( |\mathbf a|+1\right)^n  \sum_{i=1}^n \left|f_i(\mathbf a)-f_i(\mathbf b)\right|.
		\end{equation}
		The desired estimate then follows from \eqref{qb01} and \eqref{qb02}.
		
	\end{proof}

	In the rest of this paper, we only need Lemma \ref{abn} in the cases $n=2$ and $n=3$, which we state as the following two corollaries.
	
	\begin{corollary}\label{ab2}
		Let $f_1,f_2$ be the elementary symmetric polynomials in two variables.
		Let \(\mathbf a = (a_1,a_2) \in \mathbb{R}^2\), and fix \(k \in \{1,2\}\). Then there exist $\delta=\delta(\mathbf a)>0$,  and $C=C(|\mathbf a|)>0$, such that for any \(\mathbf b \in \mathbb{R}^2\) satisfying \(|b_k - a_k| < \delta\), the following hold:
		
		\begin{itemize}
			\item[(i)] If \(a_1 \ne a_2\), then
			\[
			\left|b_k - a_k\right|
			\le
			\frac{C}{|a_1 - a_2|}
			\sum_{i=1}^2 \left|f_i(\mathbf a) - f_i(\mathbf b)\right|.
			\]
			
			\item[(ii)] If \(a_1 = a_2\), then
			\[
			\left|b_k - a_k\right|^2
			\le
			C\sum_{i=1}^2 \left|f_i(\mathbf a) - f_i(\mathbf b)\right|.
			\]
		\end{itemize}
	\end{corollary}

	\begin{corollary}\label{ab3}
		
		Let $f_1,f_2,f_3$ be the elementary symmetric polynomials in three variables.
		Let \(\mathbf a = (a_1,a_2,a_3) \in \mathbb{R}^3\), and fix \(k \in \{1,2,3\}\). Then there exist $\delta=\delta(\mathbf a)>0$,  and $C=C(|\mathbf a|)>0$, such that for any \(\mathbf b \in \mathbb{R}^3\) satisfying \(|b_k - a_k| < \delta\), the following hold:
		
		\begin{itemize}
			\item[(i)] If \(a_k\) appears exactly once in the multiset $\{a_1,a_2,a_3\}$, then
			\[
			\left|b_k - a_k\right|
			\le
			\frac{C}{\left|a_i - a_k\right|\,\left|a_j - a_k\right|}
			\sum_{l=1}^3 \left|f_l(\mathbf a) - f_l(\mathbf b)\right|,
			\]
			where $\{i,j\}=\{1,2,3\}\setminus \{k\}$.
			
			\item[(ii)] If \(a_k\) appears exactly twice in the multiset $\{a_1,a_2,a_3\}$, then
			\[
			\left|b_k - a_k\right|^2
			\le
			\frac{C}{\left| a_i - a_k\right| }
			\sum_{l=1}^3 \left| f_l(\mathbf a) - f_l(\mathbf b) \right|,
			\]
			where $i\in\{1,2,3\}$ is the unique index such that $a_i \neq a_k$.
			
			\item[(iii)] If  $a_1=a_2=a_3$, then
			\[
			\left|b_k - a_k\right|^3
			\le
			C
			\sum_{i=1}^3 \left|f_i(\mathbf a) - f_i(\mathbf b)\right|.
			\]
		\end{itemize}
		
	\end{corollary}

	\section{Proof of main theorem}\label{sec4}
	
	For simplicity, write \(\mathbf O = \mathbf O_{\bar\omega}\), where \(\bar\omega\) is as in Theorem \ref{thm6d}. Let \(\omega_t := \omega(t,\cdot)\) be a mean-zero \(C^1\) solution of the Euler equation \eqref{euler}. Our goal is to estimate \({\rm dist}_2(\omega_t,\mathbf O)\).
  Denote by \(v_t\) the orthogonal projection of \(\omega_t\) onto \(\mathbf E_1\). Then
	\[
	{\rm dist}_2(\omega_t,\mathbf O)^2
	=
	{\rm dist}_2(\omega_t,\mathbf E_1)^2
	+
	{\rm dist}_2(v_t,\mathbf O)^2.
	\]
	In view of Proposition \ref{prop21}, it suffices to estimate \({\rm dist}_2(v_t,\mathbf O)\).
	To this end, we proceed in several steps.

	\subsection{Reduction to estimates of amplitudes and phases}
	Assume that $v_t$ has the form
	\begin{equation}
		v_t(\mathbf x)=\sum_{i=1}^{3} B_i(t) \cos\left(\mathbf k_i \cdot \mathbf x + \beta_i(t)\right),\qquad B_i(t)\geq 0, \quad \beta_i(t)\in\mathbb R.
	\end{equation}
	Recall that \(\mathbf O\) is determined by \(A_1, A_2, A_3\) and \(\alpha_1, \alpha_2, \alpha_3\), as stated in Theorem \ref{thm6d}.  The following lemma shows that \({\rm dist}_2(v_t,\mathbf O)\) can be expressed explicitly in terms of the amplitudes \(A_i, B_i(t)\) and the phases \(\alpha_i, \beta_i(t)\), where $i=1,2,3$.

	Throughout this section, for notational simplicity, we use $C$ to denote various positive constants \emph{depending only on $\Lambda$ and  $\|\bar\omega\|_{L^2(\mathbb{T}_\Lambda)}$}, whose values may vary from line to line. We also sometimes write \(B_i = B_i(t)\) and \(\beta_i = \beta_i(t)\) when convenient.
	
	\begin{lemma}
		It holds that
		\begin{equation}\label{min}
			\mathrm{dist}_2\bigl(v_t,\mathbf O\bigr)^2
			=C \min_{\theta_1, \theta_2 \in \mathbb{R}} \left(
			\left|B_1- A_1 e^{i\theta_1}\right|^2 +
			\left|B_2-  A_2 e^{i\theta_2}\right|^2 +
			\left|B_3- A_3 e^{i(\theta_1 + \theta_2 +\beta-\alpha)}\right|^2\right),
		\end{equation}
		where
		\begin{equation}\label{deofab}
			\alpha:=\alpha_1 + \alpha_2 - \alpha_3,\quad  \beta:=\beta_1 + \beta_2 - \beta_3.
		\end{equation}
	\end{lemma}

	\begin{proof}
		Using the fact that $\cos(\mathbf k_i\cdot \mathbf x)$ and $\sin(\mathbf k_i\cdot \mathbf x)$, $i=1,2,3$ are mutually orthogonal in $\mathring L^2(\mathbb T_\Lambda)$, we can calculate as follows:
		\begin{equation}\notag
			\begin{split}
				&\mathrm{dist}_2\bigl(v_t,\mathbf O\bigr)^2\\
				=&\min_{\mathbf p\in\mathbb T_{\Lambda}}\|v_t(\mathbf x)-\bar\omega(\mathbf x-\mathbf p) \|^2_{L^2(\mathbb T_\Lambda)}\\
				=&\min_{\mathbf p\in\mathbb T_{\Lambda}}\int_{\mathbb T_{\Lambda}}\left(\sum_{i=1}^{3} B_i\cos\left(\mathbf k_i \cdot \mathbf x + \beta_i\right)-\sum_{i=1}^{3} A_i \cos\left(\mathbf k_i \cdot (\mathbf x-\mathbf p) + \alpha_i\right)\right)^2 \dd\mathbf x\\
				=&\min_{\theta_1, \theta_2 \in \mathbb{R}}\int_{\mathbb T_{\Lambda}}\left(B_1 \cos\left(\mathbf k_1 \cdot \mathbf x + \beta_1\right)- A_1 \cos\left(\mathbf k_1 \cdot \mathbf x+\theta_1+ \alpha_1\right)\right)^2\\
				&\quad\quad+\left(B_2 \cos\left(\mathbf k_2 \cdot \mathbf x + \beta_2\right)-A_2 \cos\left(\mathbf k_2 \cdot \mathbf x+\theta_2
				+ \alpha_2\right)\right)^2 \\
				&\quad\quad+\left(B_3 \cos\left(\mathbf k_3 \cdot \mathbf x + \beta_3\right)-A_3 \cos\left(\mathbf k_3 \cdot \mathbf x+\theta_1+\theta_2
				+ \alpha_3\right)\right)^2 \dd\mathbf x\\
				=&C \min_{\theta_1, \theta_2 \in \mathbb{R}}\left(B_1 \cos\beta_1-A_1 \cos\left(\theta_1
				+  \alpha_1\right)\right)^2+\left(B_1 \sin\beta_1- A_1 \sin\left(\theta_1
				+ \alpha_1\right)\right)^2\\
				&\quad\quad+\left(B_2 \cos\beta_2- A_2 \cos\left(\theta_2
				+  \alpha_2\right)\right)^2+\left(B_2 \sin\beta_2- A_2 \sin\left(\theta_2
				+ \alpha_2\right)\right)^2\\
				&\quad\quad+\left(B_3\cos\beta_3- A_3 \cos\left(\theta_1+\theta_2
				+  \alpha_3\right)\right)^2+\left(B_3 \sin\beta_3-A_3 \sin\left(\theta_1+\theta_2
				+ \alpha_3\right)\right)^2\\
				=&C \min_{\theta_1, \theta_2 \in \mathbb{R}}\left(
				\left|B_1 e^{i\beta_1}-A_1e^{i(\theta_1+\alpha_1)}\right|^2+
				\left|B_2 e^{i\beta_2}-A_2e^{i(\theta_2+\alpha_2)}\right|^2+
				\left|B_3 e^{i\beta_3}-A_3e^{i(\theta_1+\theta_2+\alpha_3)}\right|^2\right)\\
				=&C \min_{\theta_1, \theta_2 \in \mathbb{R}} \left(
				\left|B_1-A_1 e^{i\theta_1}\right|^2 +
				\left|B_2-A_2 e^{i\theta_2}\right|^2 +
				\left|B_3-A_3 e^{i(\theta_1 + \theta_2 +\beta-\alpha)}\right|^2\right).
			\end{split}
		\end{equation}
	\end{proof}


	\subsection{Casimirs and estimates involving elementary symmetric polynomials}

	Given that ${\rm dist}_2(\omega_0,\mathbf O)<\varepsilon$, by replacing $\bar\omega$ with another element of $\mathbf O$ if necessary, we may assume  that
	\[{\rm dist}_2(\omega_0,\mathbf O)= \|\omega_0-\bar\omega\|_{L^2(\mathbb T_\Lambda)}<\varepsilon.\]
	Without loss of generality,  we also assume that $0<\varepsilon<1.$

	\begin{lemma}\label{lemlip}
		Let $k$ be a positive integer. Then
		\[
		\left|\mathsf C_k(v_t)-\mathsf C_k(\bar\omega)\right|<C_k\varepsilon,\quad \forall\,t\in\mathbb R,
		\]
		where $\mathsf C_k$ is the $k$-th order Casimir given by \eqref{kcm}, and $C_k>0$  depends on  $k,\Lambda$ and $\|\bar\omega\|_{L^2(\mathbb T_\Lambda)}.$
	\end{lemma}
	\begin{proof}
		Let $\chi\in C^\infty(\mathbb R)$ be an even cut-off satisfying
		\[
		\chi(s)=
		\begin{cases}
			1, & \mbox{if } |s|\leq 1, \\
			0, & \mbox{if } |s|\geq 2,
		\end{cases}\quad\mbox{and}\quad \|\chi'\|_{L^\infty(\mathbb R)}\leq 2.\]
		Set
		\[
		M:=
		C \left(\|\bar\omega\|_{L^2(\mathbb T_\Lambda)}+1\right),
		\]
		where $C$ is chosen such that
		\begin{equation}\label{fdneq}
			\|v\|_{L^\infty(\mathbb T_\Lambda)}
			\le C\|v\|_{L^2(\mathbb T_\Lambda)},
			\qquad \forall\, v\in\mathbf E_1.
		\end{equation}
		Note that such $C$ exists since $\mathbf E_1$ is finite-dimensional. Define the modified $k$-th Casimir by
		\begin{equation}\label{MCK}
			\tilde{\mathsf C}_k(\omega)
			:=
			\int_{\mathbb T_\Lambda}
			\omega^k\,
			\chi \left(\frac{\omega}{M}\right)\dd{\mathbf x},
		\end{equation}
		which is conserved under the Euler dynamics.
		Then, by a straightforward computation, for any $v_1,v_2\in L^2(\mathbb T_\Lambda)$,
		\begin{equation}\label{CK}
			\begin{split}
				|\tilde {\mathsf C}_k(v_1)-\tilde  {\mathsf C}_k(v_2)|&=\left|\int_{\mathbb T_\Lambda}
				h(v_1)-h(v_2)\dd{\mathbf x} \right|\quad\quad\left(h(s):=s^k
				\chi \left(\frac{s}{M}\right)\right)\\
				&\leq \|h'\|_{L^\infty(\mathbb R)}\|v_1-v_2\|_{L^1(\mathbb T_\Lambda)}\\
				&\le(2M)^{k-1}(k+4) \left|\mathbb T_{\Lambda}\right|^{1/2} \|v_1-v_2\|_{L^2(\mathbb T_\Lambda)}.
			\end{split}
		\end{equation}
		Using \eqref{CK}, Proposition \ref{prop21}, and the fact that $\tilde{\mathsf C}_k$ is conserved,
		\[
		\begin{aligned}
			|\tilde  {\mathsf C}_k(v_t)-\tilde  {\mathsf C}_k(\bar\omega)|
			&\le
			\left|\tilde  {\mathsf C}_k(v_t)-\tilde  {\mathsf C}_k(\omega_t)\right|
			+
			\left|\tilde  {\mathsf C}_k(\omega_t)-\tilde  {\mathsf C}_k(\bar\omega)\right| \\
			&=
			\left|\tilde  {\mathsf C}_k(v_t)-\tilde  {\mathsf C}_k(\omega_t)\right|
			+
			\left|\tilde  {\mathsf C}_k(\omega_0)-\tilde  {\mathsf C}_k(\bar\omega)\right| \\
			&\le
			(2M)^{k-1}(k+4) \left|\mathbb T_{\Lambda}\right|^{1/2} \left(\|v_t-\omega_t\|_{L^2(\mathbb T_\Lambda)}
			+
			\|\omega_0-\bar\omega\|_{L^2(\mathbb T_\Lambda)}\right) \\
			&\leq (2M)^{k-1}(k+4) \left|\mathbb T_{\Lambda}\right|^{1/2} \left(\sqrt{\frac{\lambda_2}{\lambda_2-\lambda_1}}\|v_0-\omega_0\|_{L^2(\mathbb T_\Lambda)}
			+
			\|\omega_0-\bar\omega\|_{L^2(\mathbb T_\Lambda)}\right) \\
			&\le
			(2M)^{k-1}(k+4) \left|\mathbb T_{\Lambda}\right|^{1/2} \left(1+\sqrt{\frac{\lambda_2}{\lambda_2-\lambda_1}}\right)
			\|\omega_0-\bar\omega\|_{L^2(\mathbb T_\Lambda)}.
		\end{aligned}
		\]
		Thus
		\[
		\left|\tilde  {\mathsf C}_k(v_t)-\tilde  {\mathsf C}_k(\bar\omega)\right|
		<
		C_k\varepsilon,
		\quad
		C_k:=(2M)^{k-1}(k+4) \left|\mathbb T_{\Lambda}\right|^{1/2}\left(1+\sqrt{\frac{\lambda_2}{\lambda_2-\lambda_1}}\right).
		\]
		To conclude the proof, it remains to show that
		\[
		\tilde  {\mathsf C}_k(\bar\omega)= {\mathsf C}_k(\bar\omega),\qquad \tilde  {\mathsf C}_k(v_t)=  {\mathsf C}_k(v_t).
		\]
		In view of the definition of $\tilde  {\mathsf C}_k$ and the fact that  $\chi\equiv 1$ in $[-M,M]$, it suffices to prove that
		\begin{equation}\label{tdeq}\|\bar\omega\|_{L^\infty(\mathbb T_\Lambda)}\leq M,\qquad \|v_t\|_{L^\infty(\mathbb T_\Lambda)}\leq M. \end{equation}
		For $\bar\omega$,
		by the definition of $M$,
		\[
		\|\bar\omega\|_{L^\infty(\mathbb T_\Lambda)}\leq C\|\bar\omega\|_{L^2(\mathbb T_\Lambda)}\leq M.
		\]
		As to  $v_t$, recalling that  $v_t$ is the orthogonal projection of $\omega_t$ onto $\mathbf E_1$, we have
		\begin{equation}\label{vtbdd}
			\begin{split}
				\|v_t\|_{L^2(\mathbb T_\Lambda)}
				\le \|\omega_t\|_{L^2(\mathbb T_\Lambda)}
				= \|\omega_0\|_{L^2(\mathbb T_\Lambda)}
				&\le \|\bar\omega\|_{L^2(\mathbb T_\Lambda)}
				+\|\omega_0-\bar\omega\|_{L^2(\mathbb T_\Lambda)}\\
				&< \|\bar\omega\|_{L^2(\mathbb T_\Lambda)}
				+\varepsilon\\
				& < \|\bar\omega\|_{L^2(\mathbb T_\Lambda)}
				+1.
			\end{split}
		\end{equation}
		Therefore,
		\[
		\|v_t\|_{L^\infty(\mathbb T_\Lambda)}
		\le
		C \|v_t\|_{L^2(\mathbb T_\Lambda)}
		<
		C \bigl(\|\bar\omega\|_{L^2(\mathbb T_\Lambda)}+1\bigr)
		= M.
		\]
		Hence \eqref{tdeq} has been verified, and the proof is complete.
	\end{proof}
	
	Let $f_1,f_2,f_3$ be the elementary symmetric polynomials in three variables. Denote
	\[\mathbf a=(A_1^2,A_2^2,A_3^2),\qquad \mathbf b=(B_1^2,B_2^2,B_3^2).\]
	From Lemma \ref{lemlip}, we can prove the following lemma.
	\begin{lemma}\label{c3-f-control}
		For any $t\in\mathbb R,$ it holds that
		\begin{equation}\label{C3lip}
			\bigl|B_1B_2B_3\cos\beta-A_1A_2A_3\cos\alpha \bigr|< C\varepsilon
		\end{equation}
		and
		\begin{equation}\label{f-control}
			\left|f_i(\mathbf a)-f_i(\mathbf b)\right|< C\varepsilon,
			\qquad i=1,2,3.
		\end{equation}
		
	\end{lemma}
	
	\begin{proof}
		We compute the Casimirs $ {\mathsf C}_k(\bar \omega)$ for $k=2,3,4,6$ with the help of Maple:
		\begin{equation}\label{C2346}
			\begin{split}
				& {\mathsf C}_2(\bar \omega)=
				\frac{|\mathbb T_\Lambda|}{2}\left(A_1^2+A_2^2+A_3^2\right),\\
				& {\mathsf C}_3(\bar \omega)
				=\frac{3\left|\mathbb T_\Lambda\right|}{2} A_1 A_2 A_3\cos\alpha,\\
				& {\mathsf C}_4(\bar \omega)
				=\frac{3\left|\mathbb T_\Lambda\right|}{8} \left({A_1^4}+{A_2^4}+{A_3^4}+{4A_1^2A_2^2}+{4A_1^2A_3^2}
				+{4A_2^2A_3^2} \right),\\
				& {\mathsf C}_6(\bar \omega)
				=\frac{5\left|\mathbb T_\Lambda\right|}{16}\Bigl({A_1^6}+{A_2^6}+{A_3^6}+{9A_2^2A_3^4}+{9A_2^4A_3^2}
				+{9A_1^4A_2^2}+{9A_1^4A_3^2}
				+{9A_1^2A_2^4}+{9A_1^2A_3^4}\\
				&\qquad\qquad
				+{9A_1^2A_2^2A_3^2\cos(2\alpha)}
				+{36A_1^2A_2^2A_3^2}
				\Bigr).
			\end{split}
		\end{equation}
		Note that taking $k=5$ does not give an independent constraint. A similar computation holds for $\mathsf C_k(v_t)$ with $A_1,A_2,A_3,\alpha$ replaced by 		$B_1,B_2,B_3,\beta$. Then \eqref{C3lip} follows by applying Lemma \ref{lemlip} with $k=3$.  To prove \eqref{f-control}, we observe that from \eqref{C2346},
		\begin{equation}\label{f123b}
			\begin{cases}
				f_1(\mathbf a)= \frac{2}{\left|\mathbb T_\Lambda\right|} {\mathsf C}_2(\bar \omega), \\[1ex]
				f_1^2(\mathbf a)+2f_2(\mathbf a)= \frac{8}{3\left|\mathbb T_\Lambda\right|} {\mathsf C}_4(\bar \omega), \\[1ex]
				f_1^3(\mathbf a)+6f_1(\mathbf a)f_2(\mathbf a)+3f_3(\mathbf a)
				=
				\frac{16}{5\left|\mathbb T_\Lambda\right|} {\mathsf C}_6(\bar \omega)-\frac{8}{\left|\mathbb T_\Lambda\right|^2} {\mathsf C}_3(\bar \omega)^2,
			\end{cases}
		\end{equation}
		and
		\begin{equation}\label{f123t}
			\begin{cases}
				f_1(\mathbf b)= \frac{2}{\left|\mathbb T_\Lambda\right|} {\mathsf C}_2(v_t), \\[1ex]
				f_1^2(\mathbf b)+2f_2(\mathbf b)= \frac{8}{3\left|\mathbb T_\Lambda\right|} {\mathsf C}_4(v_t), \\[1ex]
				f_1^3(\mathbf b)+6f_1(\mathbf b)f_2(\mathbf b)+3f_3(\mathbf b)
				=
				\frac{16}{5\left|\mathbb T_\Lambda\right|} {\mathsf C}_6(v_t)-\frac{8}{\left|\mathbb T_\Lambda\right|^2} {\mathsf C}_3(v_t)^2.
			\end{cases}
		\end{equation}
		Based on \eqref{f123b}, \eqref{f123t} and Lemma \ref{lemlip}, we analyze as follows:
		\begin{itemize}
			\item [(i)] From \eqref{f123b}$_1$ and \eqref{f123t}$_1$, we get \eqref{f-control} with $i=1$.	
			\item [(ii)] From \eqref{f123b}$_2$ and \eqref{f123t}$_2$, we get
			\begin{align*}
				\left|f_2(\mathbf a)-f_2(\mathbf b)\right|&< C \varepsilon+C\left|f_1^2 (\mathbf a)-f_1^2 (\mathbf b)\right| \\
				&< C\varepsilon+C\varepsilon(\left|\mathbf a\right|+\left|\mathbf b\right|)\\
				&< C\varepsilon,
			\end{align*}
			where we used  the estimate for $i=1$, and
			\begin{equation}\label{ababsv}
				|\mathbf a|\leq C\|\bar\omega\|^2_{L^2(\mathbb T_\Lambda)},\quad |\mathbf b|\leq C\|v_t\|^2_{L^2(\mathbb T_\Lambda)}\leq C\left(\|\bar\omega\|^2_{L^2(\mathbb T_\Lambda)}+1\right) \quad\mbox{(by \eqref{vtbdd})}.
			\end{equation}
			\item [(iii)] From \eqref{f123b}$_3$ and \eqref{f123t}$_3$, we get
			\begin{align*}
				\left|f_3(\mathbf a)-f_3(\mathbf b)\right|
				&< C \varepsilon+ C\varepsilon\left|\mathsf C_3(\bar\omega)+\mathsf C_3(v_t)\right|+C|f_1 (\mathbf a)f_2(\mathbf a)-f_1 (\mathbf b)f_2(\mathbf b)|+C|f^3_1 (\mathbf a) -f^3_1 (\mathbf b)|\\
				&< C \varepsilon+ C\varepsilon\left(\|\bar\omega\|^3_{L^\infty(\mathbb T_\Lambda)}+\|v_t\|^3_{L^\infty(\mathbb T_\Lambda)}+|\mathbf a|+|\mathbf b|^2+|\mathbf a|^2+|\mathbf b|^4\right)\\
				&< C\varepsilon,
			\end{align*}
			where we used \eqref{fdneq}, \eqref{ababsv} and the estimates for $i=1,2$.
		\end{itemize}
		The proof is complete.
	\end{proof}

	\subsection{Estimate of $|b_k(t)-a_k|$}
	
	Let  $k\in\{1,2,3\}$ be fixed.
	Combining Lemma \ref{c3-f-control} with Corollary \ref{ab3}, we obtain  the following estimate of $|b_k(t)- a_k|$ for all $t\in\mathbb R$.
	\begin{lemma}\label{abk}
		There exists  $\varepsilon_0=\varepsilon_0(\Lambda,\bar\omega)>0,$ such that if $\varepsilon<\varepsilon_0$, then the following hold:
		\begin{itemize}
			\item[(i)] If $a_k$ appears exactly once in the multiset $\{a_1, a_2, a_3\}$, then
			\begin{equation}\label{ak1}
				\left|b_k(t) - a_k\right|
				<
				\frac{C \varepsilon}{\left|a_i - a_k\right| \, \left|a_j - a_k\right|},\quad\forall\,t\in\mathbb R,
			\end{equation}
			where $\{i,j\} = \{1,2,3\} \setminus \{k\}$.
			
			\item[(ii)]  If $a_k$ appears exactly twice in the multiset $\{a_1, a_2, a_3\}$, then
			\begin{equation}\label{ak2}
				\left|b_k(t) - a_k\right|
				<
				\frac{C\varepsilon^{1/2}}{\left|a_i - a_k\right|^{1/2}},\quad\forall\,t\in\mathbb R,
			\end{equation}
			where $i\in\{1,2,3\}$ is the unique index such that $a_i \neq a_k$.	
			\item[(iii)]  If $a_1=a_2=a_3$, then
			\begin{equation}\label{ak3}
				\left|b_k(t) - a_k\right|
				<
				C  \varepsilon^{1/3},\quad\forall\,t\in\mathbb R.
			\end{equation}
		\end{itemize}
	\end{lemma}
	
	\begin{proof}
		We prove only (iii), as the other two cases follow by a similar argument. By Corollary \ref{ab3} and \eqref{f-control} in Lemma \ref{c3-f-control},   there exists
		$\delta=\delta(\mathbf a)>0$ such that
		\begin{equation}\label{c1o3}
			\left|b_k(t)-a_k\right|<\delta\quad\Longrightarrow\quad  |b_k(t) - a_k|
			<
			C  \varepsilon^{1/3}
		\end{equation}
		for any $t\in\mathbb R.$
		To finish the proof, it remains to show
		\begin{equation}\label{lesdelta}
			|b_k(t)-a_k|<\delta,\quad \forall\,t\in\mathbb R,
		\end{equation}
		if $\varepsilon$ is sufficiently small. To this end, recall that
		\[
		\|v_0-\bar\omega\|_{L^2(\mathbb T_\Lambda)}
		\leq
		\|\omega_0-\bar\omega\|_{L^2(\mathbb T_\Lambda)}<\varepsilon,
		\]
		which implies that
		\begin{equation}\label{a0-close-b}
			\left|b_k(0)-a_k\right|< C \varepsilon<C\varepsilon^{1/3}.
		\end{equation}
		Take $\varepsilon_0=\varepsilon_0(\Lambda,\bar\omega)>0$   such that $C\varepsilon^{1/3}<\delta/3$ for any $\varepsilon\in(0,\varepsilon_0)$. We now verify \eqref{lesdelta} by contradiction. Suppose that  $\left|b_k(t^*)-a_k\right|\geq \delta$ at some $t^*\in\mathbb R$. Then, by continuity, there exists $t^{**}\in\mathbb R$ such that  $\left|b_k(t^{**})-a_k\right|=\delta/2,$ which in combination with \eqref{c1o3} gives   $\left|b_{k}(t^{**})-a_k\right|<C\varepsilon^{1/3}<\delta/3,$ a contradiction. The proof is complete.
		
	\end{proof}

	\subsection{Concluding the proof: $A_1A_2A_3=0$}

	Without loss of generality, we may assume that $A_3=0$. Then \eqref{min} becomes
	\[
	\mathrm{dist}_2\bigl(v_t,\mathbf O\bigr)^2
	=C\left(\left|B_1-A_1\right|^2 + \left|B_2-A_2\right|^2 + \left| B_3\right|^2\right).
	\]
	So
	\[
	\mathrm{dist}_2\bigl(v_t,\mathbf O\bigr)
	<C\left(\left|B_1-A_1\right|  + \left|B_2-A_2\right| + B_3 \right).
	\]
	Now we estimate $\left|B_1-A_1\right|$, $\left|B_2-A_2\right|$ and $B_3$ based on Lemma \ref{abk}  (recall that $\mathbf a=\left(A_1^2,A_2^2,A_3^2\right), \mathbf b=\left(B_1^2,B_2^2,B_3^2\right)$). We distinguish three cases.

\medskip
	\noindent {\bf Case 1:  $A_1A_2 \neq 0$ and $A_1 \neq A_2$.} In this case,
	\[
	\left|B_1^2 - A_1^2\right| < \frac{C\varepsilon}{\left|\left(A_2^2 - A_1^2\right)\left(A_3^2 - A_1^2\right)\right|},\quad
	\left|B_2^2 - A_2^2\right| < \frac{C\varepsilon}{\left|\left(A_1^2 - A_2^2\right)\left(A_3^2 - A_2^2\right)\right|},\quad
	B_3^2 < \frac{C\varepsilon}{A_1^2 A_2^2},
	\]
	which yields
	\[
	\left|B_1 - A_1\right| < \frac{C\varepsilon}{\left|\left(A_2^2 - A_1^2\right)\left(A_3^2 - A_1^2\right)\left(B_1 + A_1\right)\right|}<\frac{C\varepsilon}{\left|\left(A_2^2 - A_1^2\right)\left(A_3^2 - A_1^2\right)\right| A_1 },\]
	\[
	\left|B_2 - A_2\right| < \frac{C\varepsilon}{\left|\left(A_1^2 - A_2^2\right)\left(A_3^2 - A_2^2\right)\left(B_2 + A_2\right)\right|}<\frac{C\varepsilon}{\left|\left(A_1^2 - A_2^2\right)\left(A_3^2 - A_2^2\right)\right| A_2 },
	\] and
	\[B_3 < \frac{C\varepsilon^{1/2}}{A_1 A_2 }.\]
	Thus
	\begin{equation*}
		\begin{split}
			\mathrm{dist}_2\bigl(v_t,\mathbf O\bigr)
			&<\frac{C  \varepsilon }{\left|\left(A_2^2-A_1^2\right) \left(A_3^2-A_1^2\right)\right|
				A_1 }+\frac{C  \varepsilon }{\left|\left(A_1^2-A_2^2\right) \left(A_3^2-A_2^2\right)\right|
				A_2 }+\frac{C  \varepsilon^{1/2}}{ A_1A_2}\\
			&\leq \frac{C  \varepsilon^{1/2}}{ A_1A_2}.
		\end{split}
	\end{equation*}
	Note that we may need to choose a smaller $\varepsilon_0$ so that the last inequality holds.
	
	\medskip
	\noindent {\bf Case 2: $A_1A_2\neq 0$ and $A_1=A_2$.} In this case,
	\[
	\left|B_1^2-A_1^2\right|+\left|B_2^2-A_2^2\right|< \frac{C\varepsilon^{1/2}}{A_1},
	\qquad
	B_3^2< \frac{C \varepsilon}{A_1^4},
	\]
	which yields
	\[\left|B_1 -A_1\right| <  \frac{C\varepsilon^{1/2}}{A_1^2},\quad
	\left|B_2 -A_2\right| <  \frac{C\varepsilon^{1/2}}{A_1^2},\quad B_3<\frac{C\varepsilon^{1/2}}{A_1^2}.\]
	Thus
	\begin{equation*}
		\mathrm{dist}_2\bigl(v_t,\mathbf O\bigr) <\frac{C\varepsilon^{1/2}}{A_1^2}.
	\end{equation*}

\medskip
	\noindent {\bf Case 3:  $A_1A_2=0$.} Without loss of generality, we may assume that $A_2=0$. Then
	\[
	\left|B_1^2-A_1^2\right|<\frac{C \varepsilon}{A_1^4},
	\quad
	B_2^2+B_3^2< \frac{C \varepsilon^{1/2}}{A_1},
	\]
	which gives
	\[
	\left|B_1 -A_1 \right|<\frac{C \varepsilon}{A_1^5},
	\quad
	B_2 < \frac{C \varepsilon^{1/4}}{A^{1/2}_1},\quad	B_3 < \frac{C \varepsilon^{1/4}}{A^{1/2}_1}.
	\]
	Thus, by choosing a smaller $\varepsilon_0$ if necessary, we  get
	\[	\mathrm{dist}_2\bigl(v_t,\mathbf O\bigr) <\frac{C\varepsilon^{1/4}}{A_1^{1/2}}.\]
	This completes the proof of Theorem \ref{thm6d} in the case $A_1A_2A_3=0$.
	
	\subsection{Concluding the proof: $A_1A_2A_3\neq 0$}
	In this case, we further estimate $\mathrm{dist}_2\bigl(v_t,\mathbf O\bigr)$ according to \eqref{min} as follows:
	\begin{equation}\label{min_control}
		\begin{split}
			&\mathrm{dist}_2\bigl(v_t,\mathbf O\bigr)^2\\
			=&C\min_{\theta_1, \theta_2 \in \mathbb{R}} \left(
			\left|B_1-A_1 e^{i\theta_1}\right|^2 +
			\left|B_2-A_2 e^{i\theta_2}\right|^2 +
			\left|B_3-A_3 e^{i(\theta_1 + \theta_2 +\beta-\alpha)}\right|^2\right)\\
			=&C\min_{\theta_1, \theta_2 \in \mathbb{R}} \left(
			\left|B_1-A_1 e^{i\theta_1}\right|^2 +
			\left|B_2-A_2 e^{i\theta_2}\right|^2 +
			\left|B_3e^{-i(\beta-\alpha)}-A_3 e^{i(\theta_1 + \theta_2)}\right|^2\right)\\
			\leq&C\min_{\theta_1, \theta_2 \in \mathbb{R}} \left(
			\left|B_1-A_1 e^{i\theta_1}\right|^2 +
			\left|B_2-A_2 e^{i\theta_2}\right|^2  + \left(
			\left|B_3e^{-i(\beta-\alpha)}-A_3\right|+
			\left|A_3-A_3 e^{i(\theta_1 + \theta_2)}\right|\right)^2\right)\\
			=&C \left(
			\left|B_1 -A_1\right|^2 +
			\left|B_2 -A_2\right|^2 +
			\left|B_3e^{-i\beta}-A_3e^{-i\alpha}\right|^2\right).
		\end{split}
	\end{equation}
	So
	\begin{equation*}
		\mathrm{dist}_2\bigl(v_t,\mathbf O\bigr)
		\leq C \left(
		\left|B_1 -A_1\right| +
		\left|B_2 -A_2\right| +
		\left|B_3e^{-i\beta}-A_3e^{-i\alpha}\right|\right).
	\end{equation*}
	To estimate $\left|B_1-A_1\right| $, $\left|B_2 -A_2\right|$ and $\left|B_3e^{-i\beta}-A_3e^{-i\alpha}\right|$, we distinguish three cases.
	
	\medskip
	\noindent\textbf{Case 1: $A_1, A_2, A_3$  are pairwise distinct.}
	In this case, based on Lemma \ref{abk}, we have the following estimates on $|B_k-A_k|$, $k=1,2,3$:
	\begin{equation}\label{bkak1}
		\begin{split}
			&\left|B_1^2-A_1^2\right|
			<\frac{C\varepsilon}{\left|A_2^2-A_1^2\right|
				\,\left|A_3^2-A_1^2\right|}
			\quad\Longrightarrow\quad
			\left|B_1 -A_1 \right|
			<\frac{C\varepsilon}{\left|A_2^2-A_1^2\right|
				\,\left|A_3^2-A_1^2\right|A_1},\\
			&\left|B_2^2-A_2^2\right|<\frac{C\varepsilon}{\left|A_1^2-A_2^2\right|
				\,\left|A_3^2-A_2^2\right|}\quad\Longrightarrow\quad  \left|B_2 -A_2 \right|<\frac{C\varepsilon}{\left|A_1^2-A_2^2\right|
				\,\left|A_3^2-A_2^2\right|A_2},\\
			&\left|B_3^2-A_3^2\right|<\frac{C\varepsilon}{\left|A_1^2-A_3^2\right|\,\left|A_2^2-A_3^2\right|}\quad\Longrightarrow\quad  \left|B_3 -A_3 \right|<\frac{C\varepsilon}{\left|A_1^2-A_3^2\right|\,\left|A_2^2-A_3^2\right|A_3}.
		\end{split}
	\end{equation}
	To estimate $|B_3 e^{-i\beta}-A_3e^{-i\alpha}|
	$,  notice that \[|B_3 e^{-i\beta}-A_3e^{-i\alpha}|
	\leq |B_3\cos\beta-A_3\cos\alpha|+|B_3\sin\beta-A_3\sin\alpha|.\]
	So it suffices to estimate $|B_3\cos\beta-A_3\cos\alpha|$ and $|B_3\sin\beta-A_3\sin\alpha|$.
	To estimate $|B_3\cos\beta-A_3\cos\alpha|$, notice that
	\begin{align*}
		\left|B_3\cos\beta-A_3\cos\alpha\right| &=\frac{
			\left|B_1B_2B_3\cos\beta-A_1A_2A_3\cos\alpha +A_1A_2A_3\cos\alpha
			-B_1B_2A_3\cos\alpha\right|
		}{ B_1B_2 }\\
		&\leq
		\frac{
			\left|B_1B_2B_3\cos\beta-A_1A_2A_3\cos\alpha\right|
			+
			A_3\left|B_1B_2-A_1A_2\right|
		}{ B_1B_2 }\\
		&\leq  \frac{
			\left|B_1B_2B_3\cos\beta-A_1A_2A_3\cos\alpha\right|
			+
			A_3\left(\left|B_1B_2-A_1B_2\right|+\left|A_1B_2-A_1A_2\right|\right)
		}{ B_1B_2 }.
	\end{align*}
	By choosing a smaller $\varepsilon_0$ if necessary, we may assume that $A_i/2<B_i< 3A_i/2$, $i=1,2,3$. Then, in view of \eqref{C3lip}  and \eqref{bkak1},
	\begin{equation}\label{BACOS1}
		\left|B_3\cos\beta-A_3\cos\alpha\right|  \leq \frac{C\varepsilon}{A_1A_2}+\frac{CA_3\varepsilon}{A_1^2\left|A_2^2-A_1^2\right|
			\,\left|A_3^2-A_1^2\right|}
		+\frac{CA_3\varepsilon}{A_2^2\left|A_1^2-A_2^2\right| \,\left|A_3^2-A_2^2\right|}.
	\end{equation}
	We now turn to the estimate of $|B_3\sin\beta-A_3\sin\alpha|$. Notice that
	\begin{align*}
		\left| \left(B_3\sin\beta\right)^2-\left(A_3\sin\alpha\right)^2\right|
		=&\left| B_3^2-A_3^2+  \left(A_3\cos\alpha\right)^2-\left(B_3\cos\beta\right)^2 \right|\\
		\leq & \left(B_3+A_3\right) \left(\left|B_3-A_3 \right|
		+\left|B_3\cos\beta-A_3\cos\alpha \right|\right).
	\end{align*}
	Taking into account \eqref{bkak1}$_3$ and \eqref{BACOS1},
	\begin{equation}\label{b2a2}
		\begin{split}
			&\left| B_3\sin\beta-A_3\sin\alpha\right|\,\left| B_3\sin\beta+A_3\sin\alpha\right|\\
			\leq& C\left( \frac{1}{\left|A_1^2-A_3^2\right|
				\,\left|A_2^2-A_3^2\right|}
			+\frac{A_3^2}{\left|A_2^2-A_1^2\right|
				\,\left|A_3^2-A_1^2\right|A_1^2}
			+\frac{A_3^2}{\left|A_1^2-A_2^2\right|
				\,\left|A_3^2-A_2^2\right|A_2^2}+ \frac{A_3}{A_1A_2}\right)\varepsilon.
		\end{split}
	\end{equation}
	We further distinguish two cases.
	
	\medskip
	\noindent{\bf Case 1.1: $\sin\alpha=0$. }
	In this case, $|B_3\sin\beta \pm A_3\sin\alpha|
	=|B_3\sin\beta|$. Thus, in view of \eqref{b2a2},
	\begin{equation*}
		\begin{split}
			&|B_3\sin\beta-A_3\sin\alpha|\\
			\leq& C\left(
			\frac{1}{|A_1^2-A_3^2||A_2^2-A_3^2|}+
			\frac{A_3^2}{|A_2^2-A_1^2||A_3^2-A_1^2|A_1^2}
			+\frac{A_3^2}{|A_1^2-A_2^2||A_3^2-A_2^2|A_2^2}
			+\frac{A_3}{A_1A_2}
			\right)^{1/2}\varepsilon^{1/2}.
		\end{split}
	\end{equation*}
	Combining this estimate with \eqref{bkak1} and \eqref{BACOS1}, and
	choosing $\varepsilon_0>0$ smaller if necessary, we obtain
	\begin{equation*}
		\begin{split}
			&\mathrm{dist}_2\bigl(v_t,\mathbf O\bigr)\\
			\leq&C\left(
			\frac{1}{|A_1^2-A_3^2||A_2^2-A_3^2|}+
			\frac{A_3^2}{|A_2^2-A_1^2||A_3^2-A_1^2|A_1^2}
			+\frac{A_3^2}{|A_1^2-A_2^2||A_3^2-A_2^2|A_2^2}
			+\frac{A_3}{A_1A_2}
			\right)^{1/2}\varepsilon^{1/2}.
		\end{split}
	\end{equation*}

	\medskip
	\noindent{\bf Case 1.2: $\sin\alpha\neq 0$.}
	In this case, we have
	\begin{equation}\label{positive_lower_bound}
		0<2A_3|\sin\alpha|
		\leq
		\left|B_3\sin\beta-A_3\sin\alpha\right|
		+
		\left|B_3\sin\beta+A_3\sin\alpha\right|.
	\end{equation}
	Combining \eqref{b2a2} with \eqref{positive_lower_bound}, we obtain the following pointwise-in-time dichotomy. For every $t\in\mathbb R$, at least one of the following two alternatives holds:
	\begin{equation}\label{alternative-new-1}
		\begin{aligned}
			&\left|B_3\sin\beta-A_3\sin\alpha\right|\\
			\leq&
			\frac{C\varepsilon}{|\sin\alpha|}
			\Bigg(
			\frac{1}
			{A_3|A_1^2-A_3^2||A_2^2-A_3^2|}
			+
			\frac{A_3}
			{|A_2^2-A_1^2||A_3^2-A_1^2|A_1^2}
			+
			\frac{A_3}
			{|A_1^2-A_2^2||A_3^2-A_2^2|A_2^2}
			+
			\frac{1}{A_1A_2}
			\Bigg),
		\end{aligned}
	\end{equation}
	or
	\begin{equation}\label{alternative-new-2}
		\begin{aligned}
			&\left|B_3\sin\beta+A_3\sin\alpha\right|\\
			\leq&
			\frac{C\varepsilon}{|\sin\alpha|}
			\Bigg(
			\frac{1}
			{A_3|A_1^2-A_3^2||A_2^2-A_3^2|}
			+
			\frac{A_3}
			{|A_2^2-A_1^2||A_3^2-A_1^2|A_1^2}
			+
			\frac{A_3}
			{|A_1^2-A_2^2||A_3^2-A_2^2|A_2^2}
			+
			\frac{1}{A_1A_2}
			\Bigg).
		\end{aligned}
	\end{equation}
	Indeed, by \eqref{positive_lower_bound}, at least one of $\left|B_3\sin\beta-A_3\sin\alpha\right|$, $\left|B_3\sin\beta+A_3\sin\alpha\right|$ is not smaller than $A_3|\sin\alpha|$. Dividing the product estimate
	\eqref{b2a2} by this lower bound gives either \eqref{alternative-new-1}
	or \eqref{alternative-new-2}. The constants in this dichotomy are independent
	of $t$, but the dichotomy is, at this stage, only pointwise in time.
	
	We next show that, for sufficiently small $\varepsilon$, the second alternative cannot occur. If
	\eqref{alternative-new-1} holds at a time \(t\), then combining
	\eqref{bkak1}, \eqref{BACOS1}, and \eqref{alternative-new-1}, we obtain
	\begin{equation}\label{dist-O-case12}
		\begin{aligned}
			&\mathrm{dist}_2(v_t,\mathbf O)\\
			\leq\;&
			\frac{C\varepsilon}
			{|A_2^2-A_1^2||A_3^2-A_1^2|}
			\left(
			\frac{1}{A_1}
			+
			\frac{A_3}{A_1^2}
			+
			\frac{A_3}{A_1^2|\sin\alpha|}
			\right)
			+
			\frac{C\varepsilon}
			{|A_1^2-A_2^2||A_3^2-A_2^2|}
			\left(
			\frac{1}{A_2}
			+
			\frac{A_3}{A_2^2}
			+
			\frac{A_3}{A_2^2|\sin\alpha|}
			\right)\\
			&+
			\frac{C\varepsilon}
			{|A_1^2-A_3^2||A_2^2-A_3^2|A_3|\sin\alpha|}
			+
			\frac{C\varepsilon}{A_1A_2}
			+
			\frac{C\varepsilon}{A_1A_2|\sin\alpha|}.
		\end{aligned}
	\end{equation}
	On the other hand, if \eqref{alternative-new-2} holds at a time \(t\),
	then the same argument, applied to the conjugate phase configuration, gives
	\begin{equation}\label{dist-conj-case12}
		\begin{aligned}
			&\mathrm{dist}_2(v_t,\widetilde{\mathbf O})\\
			\leq\;&
			\frac{C\varepsilon}
			{|A_2^2-A_1^2||A_3^2-A_1^2|}
			\left(
			\frac{1}{A_1}
			+
			\frac{A_3}{A_1^2}
			+
			\frac{A_3}{A_1^2|\sin\alpha|}
			\right)
			+
			\frac{C\varepsilon}
			{|A_1^2-A_2^2||A_3^2-A_2^2|}
			\left(
			\frac{1}{A_2}
			+
			\frac{A_3}{A_2^2}
			+
			\frac{A_3}{A_2^2|\sin\alpha|}
			\right)
			\\
			&+
			\frac{C\varepsilon}
			{|A_1^2-A_3^2||A_2^2-A_3^2|A_3|\sin\alpha|}
			+
			\frac{C\varepsilon}{A_1A_2}
			+
			\frac{C\varepsilon}{A_1A_2|\sin\alpha|},
		\end{aligned}
	\end{equation}
	where
	\[
	\widetilde{\mathbf O}
	:=
	\left\{
	\widetilde\omega(\cdot-\mathbf p)
	\mid
	\mathbf p\in\mathbb T_{\Lambda}
	\right\},
	\qquad
	\widetilde\omega(\mathbf x)
	=
	\sum_{i=1}^3 A_i\cos(\mathbf k_i\cdot\mathbf x-\alpha_i).
	\]
	Since $\sin\alpha\neq0$, $\mathbf O\neq \widetilde{\mathbf O}.$
	By Lemma B.1(3) in \cite{W252}, the two compact orbits are separated in
	\(L^2\). Thus
	\[
	d_0:=\mathrm{dist}_2(\mathbf O,\widetilde{\mathbf O})>0.
	\]
	Choose \(\varepsilon_0>0\) sufficiently small such that, whenever
	\(0<\varepsilon<\varepsilon_0\), the right-hand sides of both
	\eqref{dist-O-case12} and \eqref{dist-conj-case12} are smaller than
	\(d_0/3\). We also choose \(\varepsilon_0\) so that
	\[
	\mathrm{dist}_2(v_0,\mathbf O)<\frac{d_0}{3}.
	\]
	
	Define
	\[
	I_1:=
	\left\{
	t\in\mathbb R:
	\mathrm{dist}_2(v_t,\mathbf O)<\frac{d_0}{3}
	\right\},
	\qquad
	I_2:=
	\left\{
	t\in\mathbb R:
	\mathrm{dist}_2(v_t,\widetilde{\mathbf O})<\frac{d_0}{3}
	\right\}.
	\]
	The pointwise dichotomy above implies that $I_1\cup I_2=\mathbb R.$
	Moreover,
	\[
	I_1\cap I_2=\varnothing.
	\]
	Indeed, if \(t\in I_1\cap I_2\), then
	\[
	d_0
	=
	\mathrm{dist}_2(\mathbf O,\widetilde{\mathbf O})
	\leq
	\mathrm{dist}_2(v_t,\mathbf O)
	+
	\mathrm{dist}_2(v_t,\widetilde{\mathbf O})
	<
	\frac{2d_0}{3},
	\]
	which is impossible. Since \(t\mapsto v_t\) is continuous in \(L^2\), both
	\(I_1\) and \(I_2\) are open subsets of \(\mathbb R\). Furthermore,
	\(0\in I_1\). Since \(\mathbb R\) is connected, we must have
	\[
	I_2=\varnothing.
	\]
	Therefore the second alternative \eqref{alternative-new-2} never occurs, and
	\eqref{alternative-new-1} holds for all $t\in\mathbb R$.
	Combining this with \eqref{bkak1} and \eqref{BACOS1}, we conclude that
	\eqref{dist-O-case12} holds for all $t\in\mathbb R$.

	\medskip
	\noindent\textbf{Case 2: Exactly two of $A_1, A_2, A_3$ are equal.}
	Without loss of generality, we may assume that $A_1=A_2$, and then, based on Lemma \ref{abk}, we have the following estimates on $|B_k-A_k|$, $k=1,2,3$.
	\begin{equation}\label{bkak_2}
		\begin{split}
			&\left|B_1^2-A_1^2\right|
			<\frac{C\varepsilon^{1/2}}{\left|A_3^2-A_1^2\right|^{1/2}}
			\quad\Longrightarrow\quad
			\left|B_1 -A_1 \right|
			<\frac{C\varepsilon^{1/2}}{\left|A_3^2-A_1^2\right|^{1/2}A_1},\\
			&\left|B_2^2-A_2^2\right|<\frac{C\varepsilon^{1/2}}{\left|A_3^2-A_1^2\right|^{1/2}}\quad\Longrightarrow\quad  \left|B_2 -A_2 \right|<\frac{C\varepsilon^{1/2}}{\left|A_3^2-A_1^2\right|^{1/2}A_1},\\
			&\left|B_3^2-A_3^2\right|<\frac{C\varepsilon}{\left|A_1^2-A_3^2\right|^2}\quad\Longrightarrow\quad  \left|B_3 -A_3 \right|<\frac{C\varepsilon}{\left|A_1^2-A_3^2\right|^2A_3}.
		\end{split}
	\end{equation}
	The same estimate as in the previous case, with the bounds in \eqref{C3lip} and
	\eqref{bkak_2} substituted, gives

	\begin{equation}\label{BACOS_2}
		\begin{split}
			|B_3\cos\beta-A_3\cos\alpha|
			\leq \frac{C\varepsilon}{A_1^2}
			+\frac{CA_3\varepsilon^{1/2}}{A_1^2\left|A_3^2-A_1^2\right|^{1/2}}.
		\end{split}
	\end{equation}
	The same argument as above gives
	\begin{equation}\label{b2a2_2}
		\left|B_3\sin\beta-A_3\sin\alpha\right|\,\left|B_3\sin\beta+A_3\sin\alpha\right|
		\leq \frac{C A_3^2\varepsilon^{1/2}}{A_1^2\left|A_3^2-A_1^2\right|^{1/2}}
		+\frac{C\varepsilon}{\left|A_1^2-A_3^2\right|^2}+ \frac{CA_3\varepsilon}{A_1^2}.
	\end{equation}
	
	\medskip
	\noindent{\bf Case 2.1: $\sin\alpha=0$. } By \eqref{b2a2_2},
	\begin{equation*}
		\begin{split}
			|B_3\sin\beta-A_3\sin\alpha|
			\leq
			\frac{C A_3\varepsilon^{1/4}}{A_1\left|A_3^2-A_1^2\right|^{1/4}}
			+\frac{C\varepsilon^{1/2}}{\left|A_1^2-A_3^2\right|}+ \frac{CA_3^{1/2}\varepsilon^{1/2}}{A_1}.
		\end{split}
	\end{equation*}
	Choosing $\varepsilon_0>0$ smaller if necessary, we have
	\begin{equation*}
		\begin{split}
			\mathrm{dist}_2\bigl(v_t,\mathbf O\bigr)
			\leq & \frac{CA_3\varepsilon^{1/4}}{\left|A_3^2-A_1^2\right|^{1/4}
				A_1}.
		\end{split}
	\end{equation*}
	
	\medskip
	\noindent{\bf Case 2.2: $\sin\alpha\neq 0$.} As in Case 1.2, the pointwise dichotomy obtained from \eqref{b2a2_2} and \eqref{positive_lower_bound} has two branches. The branch corresponding to the conjugate orbit is excluded by the
	continuity of $t\mapsto v_t$ in $L^2$, the positive distance between
	$\mathbf O$ and the conjugate orbit, and the initial closeness to
	$\mathbf O$.
	From \eqref{b2a2_2}, we have that
	\begin{equation*}
		\left|B_3\sin\beta-A_3\sin\alpha\right|
		\leq \frac{C A_3\varepsilon^{1/2}}{A_1^2\left|A_3^2-A_1^2\right|^{1/2}\left|\sin\alpha\right|}
		+\frac{C\varepsilon}{\left|A_1^2-A_3^2\right|^2A_3\left|\sin\alpha\right|}+ \frac{C\varepsilon}{A_1^2\left|\sin\alpha\right|}.
	\end{equation*}
	Then,
	choosing $\varepsilon_0>0$ smaller if necessary, we obtain
	\begin{equation*}
		\begin{split}
			\mathrm{dist}_2\bigl(v_t,\mathbf O\bigr)
			\leq  \frac{C\varepsilon^{1/2}}{\left|A_3^2-A_1^2\right|^{1/2}}\left(\frac{A_3}{A_1^2\left|\sin\alpha\right|}+\frac{A_3}{A_1^2}+\frac{1}{A_1}\right).
		\end{split}
	\end{equation*}

	\medskip
	\noindent\textbf{Case 3: $A_1=A_2=A_3$.}
	Based on Lemma \ref{abk}, we have
	\begin{equation}\label{bkak_3}
		\begin{split}
			&\left|B_k^2-A_k^2\right|
			<{C\varepsilon^{1/3}}
			\quad\Longrightarrow\quad
			\left|B_k -A_k \right|
			<\frac{C\varepsilon^{1/3}}{A_1},\quad k=1,2,3.
		\end{split}
	\end{equation}
	Repeating the argument leading to \eqref{BACOS1}, we obtain
	\begin{equation}\label{BACOS_3}
		\begin{split}
			|B_3\cos\beta-A_3\cos\alpha|
			\leq \frac{C\varepsilon}{A_1^2}
			+\frac{C\varepsilon^{1/3}}{A_1}.
		\end{split}
	\end{equation}
	Repeating the argument leading to \eqref{b2a2}, we obtain
	\begin{equation}\label{b2a2_3}
		\begin{split}
			\left|B_3\sin\beta-A_3\sin\alpha\right|\,\left|B_3\sin\beta+A_3\sin\alpha\right|
			\leq C \varepsilon^{1/3}+\frac{C\varepsilon}{A_1}.
		\end{split}
	\end{equation}
	
	\medskip
	\noindent{\bf Case 3.1: $\sin\alpha= 0$.} By \eqref{b2a2_3},
	\begin{equation*}
		\begin{split}
			|B_3\sin\beta-A_3\sin\alpha|\leq C \varepsilon^{1/6}+\frac{C\varepsilon^{1/2}}{A_1^{1/2}}.
		\end{split}
	\end{equation*}
	Therefore,
	\begin{equation*}
		\begin{split}
			\mathrm{dist}_2\bigl(v_t,\mathbf O\bigr)
			\leq  C\varepsilon^{1/6},
		\end{split}
	\end{equation*}
	after choosing $\varepsilon_0$ smaller if necessary.
	
	\medskip
	\noindent{\bf Case 3.2: $\sin\alpha\neq 0$.} As in Case 1.2, the branch corresponding to the conjugate orbit is excluded by the continuity of $t\mapsto v_t$ in $L^2$, the positive distance between $\mathbf O$ and the conjugate orbit, and the initial closeness to $\mathbf O$.
	From \eqref{b2a2_3}, we have that
	\begin{equation*}
		\begin{split}
			\left|B_3\sin\beta-A_3\sin\alpha\right|
			\leq \frac{C \varepsilon^{1/3}}{\left|\sin\alpha\right|A_1}+\frac{C \varepsilon}{\left|\sin\alpha\right|A_1^2}.
		\end{split}
	\end{equation*}
	Therefore,
	\begin{equation*}
		\begin{split}
			\mathrm{dist}_2\bigl(v_t,\mathbf O\bigr)
			\leq &\frac{C \varepsilon^{1/3}}{\left|\sin\alpha\right|\,A_1}+\frac{C \varepsilon^{1/3}}{A_1},
		\end{split}
	\end{equation*}
	after choosing $\varepsilon_0$ smaller if necessary.


	
	\begin{remark}\label{symmetry-degeneracy}
		We briefly explain how the estimates in Theorem \ref{thm6d} deteriorate as the amplitude configuration becomes more symmetric, with a further loss in the degenerate phase case $\sin\alpha=0$. Assume that ${\rm dist}_{2}(\omega_0,\mathbf O)<\varepsilon$, $A_1,A_2,A_3$ are pairwise distinct and $\sin\alpha\neq 0$. Then the proof yields
		\begin{equation}\label{ABDa}
			\begin{split}
				&\mathrm{dist}_2\bigl(\omega_t,\mathbf O\bigr)\\
				\leq &
				\frac{C\varepsilon}{\left|A_2^2-A_1^2\right|\,\left|A_3^2-A_1^2\right|}
				\left(\frac{A_3}{A_1^2}+\frac{1}{A_1}+\frac{A_3}{A_1^2\left|\sin\alpha\right|}\right)
				+
				\frac{C\varepsilon}{\left|A_2^2-A_1^2\right|\,\left|A_3^2-A_2^2\right|}
				\left(\frac{A_3}{A_2^2}+\frac{1}{A_2}+\frac{A_3}{A_2^2\left|\sin\alpha\right|}\right)\\
				&+
				\frac{C\varepsilon}{\left|A_1^2-A_3^2\right|\,\left|A_2^2-A_3^2\right|\,\left|\sin\alpha\right|A_3}
				+\frac{C\varepsilon}{A_1A_2}+\frac{C\varepsilon}{A_1A_2\left|\sin\alpha\right|}.
			\end{split}
		\end{equation}
		This estimate shows explicitly how the stability deteriorates as the amplitude configuration becomes more symmetric. Suppose that $A_2\to A_1$ while $A_3$ remains fixed in \eqref{ABDa}. Since the factor $|A_2^2-A_1^2|^{-1}$
		blows up, the preceding estimate degenerates into
		\begin{equation}\label{AABa}
			\begin{split}
				\mathrm{dist}_2\bigl(\omega_t,\mathbf O\bigr)
				\leq
				\frac{C\varepsilon^{1/2}}{\left|A_3^2-A_1^2\right|^{1/2}}
				\left(\frac{A_3}{A_1^2\left|\sin\alpha\right|}
				+\frac{A_3}{A_1^2}
				+\frac{1}{A_1}\right).
			\end{split}
		\end{equation}
		Thus the stability weakens from order $O(\varepsilon)$ to order $O(\varepsilon^{1/2})$. If $A_1$, $A_2$, and $A_3$ all become close to one another, then \eqref{AABa} degenerates further into
		\begin{equation}\label{AAAa}
			\begin{split}
				\mathrm{dist}_2\bigl(\omega_t,\mathbf O\bigr)
				\leq
				\frac{C\varepsilon^{1/3}}{\left|\sin\alpha\right|A_1}
				+\frac{C\varepsilon^{1/3}}{A_1}.
			\end{split}
		\end{equation}
		Hence the stability order drops from $O(\varepsilon^{1/2})$ to $O(\varepsilon^{1/3})$.
		
		We next explain how the phase factor $\sin\alpha$ affects the stability estimate. Examining the three cases above separately, we observe the same phenomenon: as $\sin\alpha$ tends to zero, the stability estimate deteriorates.
		When $\left|\sin\alpha\right|$ approaches $0$, the bound \eqref{ABDa} degenerates to
		\begin{equation}\label{ABD0}
			\begin{split}
				&\mathrm{dist}_2\bigl(\omega_t,\mathbf O\bigr)\\
				\leq &
				C\left(
				\frac{A_3^2}{|A_2^2-A_1^2||A_3^2-A_1^2|A_1^2}
				+\frac{A_3^2}{|A_1^2-A_2^2||A_3^2-A_2^2|A_2^2}
				+\frac{1}{|A_1^2-A_3^2||A_2^2-A_3^2|}
				+\frac{A_3}{A_1A_2}
				\right)^{1/2}\varepsilon^{1/2}.
			\end{split}
		\end{equation}
		Hence the stability order drops from $O(\varepsilon)$ to $O(\varepsilon^{1/2})$.
		And, \eqref{AABa} further degenerates to
		\begin{equation}\label{AAB0}
			\begin{split}
				\mathrm{dist}_2\bigl(\omega_t,\mathbf O\bigr)
				\leq
				\frac{CA_3\varepsilon^{1/4}}{\left|A_3^2-A_1^2\right|^{1/4}A_1}.
			\end{split}
		\end{equation}
		Accordingly, the stability order drops from $O(\varepsilon^{1/2})$ to $O(\varepsilon^{1/4})$.
		Moreover, \eqref{AAAa} reduces to
		\begin{equation}\label{AAA0}
			\begin{split}
				\mathrm{dist}_2\bigl(\omega_t,\mathbf O\bigr)
				\leq
				C\varepsilon^{1/6}.
			\end{split}
		\end{equation}
		In this most symmetric regime, the stability deteriorates further from $O(\varepsilon^{1/3})$ to $O(\varepsilon^{1/6})$.

		A similar phenomenon appears when exactly one of $A_1,A_2,A_3$ is zero. Without loss of generality, we may assume that $A_3=0$. In this case, the proof gives
		\begin{equation}\label{CI4}
			{\rm dist}_2\bigl(\omega_t,\mathbf O\bigr)
			<
			\frac{C  \varepsilon^{1/2}}{ A_1A_2}.
		\end{equation}
		If $A_2\to 0$, then \eqref{CI4} degenerates to
		\[
		{\rm dist}_2\bigl(\omega_t,\mathbf O\bigr)
		<
		\frac{C\varepsilon^{1/4}}{A_1^{1/2}}.
		\]
		This corresponds to a further deterioration of the stability estimate from order $O(\varepsilon^{1/2})$ to $O(\varepsilon^{1/4})$.
	\end{remark}

	\section{Non-hexagonal tori}\label{sec5}

In this section, we consider the remaining cases of flat two-dimensional tori for which $\dim(\mathbf E_1)=2$ or $4$. For these two cases, by \cite[Sect.~2]{W252}, we have the following:
	\begin{itemize}
		
		\item[(i)] If $\dim(\mathbf{E}_1) = 2$, then there exists some nonzero vector $\mathbf k$ such that
		\[
		\mathbf{E}_1={\rm span}\left\{\cos(\mathbf k \cdot \mathbf x), \sin(\mathbf k \cdot\mathbf x)\right\}.
		\]
		\item[(ii)] If $\dim(\mathbf{E}_1) = 4$, then there exist two linearly independent vectors $\mathbf k_1, \mathbf k_2$ satisfying $|\mathbf k_1| = |\mathbf k_2|$
		such that
		\[
		\mathbf{E}_1={\rm span}\left\{ \cos(\mathbf k_1 \cdot\mathbf x), \sin(\mathbf k_1 \cdot\mathbf x), \cos(\mathbf k_2 \cdot\mathbf x), \sin(\mathbf k_2 \cdot\mathbf x)\right\}.
		\]
	\end{itemize}

\subsection{2D case}	
	
	\begin{theorem}\label{thm2d}
		Assume that $\dim(\mathbf E_1)=2$.
		Fix $\bar{\omega} \in \mathbf E_1$ of the form
		\begin{equation}\notag
			\bar \omega(\mathbf x)=A \cos(\mathbf k\cdot \mathbf x+\alpha ),\quad A> 0,\,\,\alpha\in\mathbb R.
		\end{equation}
		Then there exist positive constants $\varepsilon_0=\varepsilon_0(\Lambda,\bar\omega)$ and  $C=C(\Lambda,A)$, such that for any $\varepsilon\in(0,\varepsilon_0)$ and any mean-zero $C^1$ solution $\omega(t,\mathbf x)$ to \eqref{euler},
		\[
		{\rm dist}_{2}(\omega(0,\cdot),\mathbf O_{\bar\omega})<\varepsilon
		\quad\Longrightarrow\quad
		{\rm dist}_{2}(\omega(t,\cdot),\mathbf O_{\bar\omega})<C\varepsilon
		\quad\text{for all } t\in\mathbb R.
		\]
	\end{theorem}
	
\begin{proof}
As in Section \ref{sec4}, denote by \(v_t\) the orthogonal projection of \(\omega_t:=\omega(t,\cdot)\) onto \(\mathbf E_1\).
Then $v_t$ can be expressed as
	\[
	v_t(\mathbf x)= B(t)\cos\bigl(\mathbf k\cdot \mathbf x + \beta(t)\bigr),\quad B(t)\ge 0,\ \beta(t)\in\mathbb R.
	\]
	 For simplicity, write \(B = B(t)\), \(\beta = \beta(t)\) and $\mathbf O = \mathbf O_{\bar\omega}$, where \(\bar\omega\) is as in Theorem \ref{thm2d}.
 Then \({\rm dist}_2(v_t,\mathbf O)\) can be expressed explicitly in terms of the amplitudes \(A, B\):
	\begin{equation}\label{501}
		\begin{split}
			\mathrm{dist}_2\bigl(v_t,\mathbf O\bigr)^2
			=&\min_{ \mathbf p\in\mathbb T_\Lambda}\|v_t(\mathbf x)-\bar\omega(\mathbf x-\mathbf p) \|^2_{L^2(\mathbb T_\Lambda)}
			=C\min_{\theta \in \mathbb{R}} \left|B-Ae^{i( \theta+\alpha-\beta)}\right|^2
			=C|B-A|^2,
		\end{split}
	\end{equation}
where \(C\)  denotes various positive constants depending only on \(\Lambda\) and \(A\).
	 By replacing $\bar\omega$ with another element of its orbit if necessary,
	we may assume that
	\[
	\mathrm{dist}_2(\omega_0,\mathbf O)
	=
	\|\omega_0-\bar\omega\|_{L^2(\mathbb T_\Lambda)}
	<\varepsilon.
	\]
	Then, Lemma \ref{lemlip} with $k=2$ (note that Lemma \ref{lemlip} actually holds for flat 2-tori of arbitrary shape) yields
	\begin{equation}\label{502}
	\left| \mathsf C_2(v_t) - \mathsf C_2(\bar \omega) \right| =\frac{|\mathbb T_\Lambda|}{2}\left|B^2- A^2\right|< C \varepsilon,
	\end{equation}
where we have used
\[\mathsf C_2(\bar \omega)= \frac{|\mathbb T_\Lambda|A^2}{2}\]
by a straightforward computation.
From \eqref{501} and \eqref{502}, we get
	\[
	\mathrm{dist}_2\bigl(v_t,\mathbf O\bigr) <C|B-A|<\frac{C\varepsilon}{B+A}< \frac{C\varepsilon}{A},
	\]
	and the desired conclusion follows.
\end{proof}	
	
\subsection{4D case}	
	\begin{theorem}\label{thm4d}
		Assume that $\dim(\mathbf E_1)=4$. Fix $\bar{\omega} \in \mathbf E_1$ of the form
		\begin{equation}\notag
			\bar \omega(\mathbf x)=\sum_{i=1}^2A_i\cos(\mathbf k_i\cdot \mathbf x+\alpha_i),\quad A_i\geq 0,\,\,\alpha_i\in\mathbb R,\,\, (A_1,A_2)\neq (0,0).
		\end{equation}
	Then there exist positive constants $\varepsilon_0=\varepsilon_0(\Lambda,\bar\omega)$ and  $C=C(\Lambda,A_1,A_2)$, such that for any $\varepsilon\in(0,\varepsilon_0)$ and any mean-zero $C^1$ solution $\omega(t,\mathbf x)$ to \eqref{euler},
		\[
		{\rm dist}_{2}(\omega(0,\cdot),\mathbf O_{\bar\omega})<\varepsilon
		\quad\Longrightarrow\quad
		{\rm dist}_{2}(\omega(t,\cdot),\mathbf O_{\bar\omega})<C\varepsilon^\gamma
		\quad\text{for all } t\in\mathbb R,
		\]
		where $\gamma=1$ if $A_1A_2\neq 0$ and $A_1 \neq A_2$, and $\gamma= {1}/{2}$ otherwise.
	\end{theorem}
	
\begin{proof}
	In this case, write
	\[
	v_t(\mathbf x)=\sum_{i=1}^{2} B_i(t) \cos\left(\mathbf k_i \cdot \mathbf x + \beta_i(t)\right),\quad B_i(t)\geq0,\ \beta_i(t)\in\mathbb R.
	\]
	Then
	\begin{equation}\notag
		\begin{split}
			\mathrm{dist}_2\bigl(v_t,\mathbf O_{\bar\omega}\bigr)^2
			=&\min_{\mathbf p\in\mathbb T_\Lambda}\|v_t(\mathbf x)-\bar\omega(\mathbf x-\mathbf p)\|_{L^2(\mathbb T_\Lambda)}^2 \\
			=&C\min_{\theta_1, \,\theta_2 \in \mathbb{R}} \left(
			\left|B_1e^{i\beta_1}-A_1 e^{i(\theta_1+\alpha_1)}\right|^2+
			\left|B_2e^{i\beta_2}-A_2 e^{i(\theta_2+\alpha_2)}\right|^2\right)\\
			=&C\left(\left|B_1- A_1\right|^2+\left|B_2- A_2\right|^2\right),
		\end{split}
	\end{equation}
which yields
	\begin{equation}\notag
		\begin{split}
			\mathrm{dist}_2\bigl(v_t,\mathbf O_{\bar\omega}\bigr)
			\leq C\left(\left|B_1- A_1\right|+\left|B_2- A_2\right|\right).
		\end{split}
	\end{equation}
	By  a straightforward computation, we have that
	\begin{equation}\label{C24}
		\begin{split}
			&\mathsf C_2(\bar \omega)= \frac{|\mathbb T_\Lambda|}{2}(A_1^2 + A_2^2);\\
			&\mathsf C_4(\bar \omega)= \frac{3|\mathbb T_\Lambda|}{8}\left( A_1^4 + 4A_1^2A_2^2 + A_2^4 \right).
		\end{split}
	\end{equation}
	From \eqref{C24}, we see that $\mathbf a=(A_1^2,A_2^2)$ satisfies the following system:
	\begin{equation}\label{4df123b}
		\begin{cases}
			f_1(\mathbf a)= \frac{2}{|\mathbb T_\Lambda|} \mathsf C_2(\bar \omega), \\
			f_1^2(\mathbf a)+2f_2(\mathbf a) =\frac{8}{3|\mathbb T_\Lambda|}\mathsf C_4(\bar \omega).
		\end{cases}
	\end{equation}
	Similarly, $\mathbf b=(B_1^2,B_2^2)$ satisfies
	\begin{equation}\label{4df123t}
		\begin{cases}
			f_1(\mathbf b) = \frac{2}{|\mathbb T_\Lambda|} \mathsf C_2(v_t), \\
			f_1^2(\mathbf b)+2f_2(\mathbf b)=\frac{8}{3|\mathbb T_\Lambda|}\mathsf C_4(v_t).
		\end{cases}
	\end{equation}
As in the proof of Theorem \ref{thm2d}, we may assume that ${\rm dist}_2(\omega_0,\mathbf O)= \|\omega_0-\bar\omega\|_{L^2(\mathbb T_\Lambda)}<\varepsilon$. Then, in view of
	 \eqref{4df123b} and \eqref{4df123t}, we can apply  Lemma \ref{lemlip} to get
	\begin{equation}\label{4d-f-control}
		\left|f_i(\mathbf a)-f_i(\mathbf b)\right|< C\varepsilon,\quad i=1,2,
	\end{equation}
	where $f_1,f_2$ are the elementary symmetric polynomials in two variables.
	As in the previous case, a continuity argument shows that, for $\varepsilon$ sufficiently small, $b_k(t)$ remains in a sufficiently small neighborhood of $a_k$ for all $t\in\mathbb R$. Hence the assumptions of Corollary \ref{ab2} are satisfied. In view of  Corollary \ref{ab2} and \eqref{4d-f-control}, we distinguish three cases:
	\begin{itemize}
		
		\item[(1)]  $A_1\neq  A_2$ and $ A_1 A_2\neq 0$. In this case,
		\[
		\left|B_i^2-A_i^2\right|<\frac{C \varepsilon}{\left|A_1^2-A_2^2\right|} \quad\Longrightarrow\quad \left|B_i-A_i\right|<\frac{C \varepsilon}{|A_1^2-A_2^2|\,A_i},
		\]
		where $i=1,2.$
		Therefore,
		\[
		\mathrm{dist}_2\bigl(v_t,\mathbf O\bigr)
		< \sum_{i=1}^2\frac{C\varepsilon}{A_i\,|A_1^2-A_2^2|};
		\]
		
		\item[(2)] Exactly one of $ A_1, A_2$ vanishes. In this case, without loss of generality, we may assume that $A_2=0$. We obtain
		\[
		|B_1^2-A_1^2|<\frac{C\varepsilon}{A_1^2}\quad\Longrightarrow\quad |B_1-A_1|<\frac{C\varepsilon}{A_1^3},
		\qquad
		|B_2|^2< \frac{C\varepsilon}{A_1^2}\quad\Longrightarrow\quad|B_2|< \frac{C\varepsilon^{1/2}}{A_1}.
		\]
		Hence
		\[
		\mathrm{dist}_2\bigl(v_t,\mathbf O\bigr)
		< \frac{C\varepsilon}{A_1^3}+\frac{C\varepsilon^{1/2}}{A_1}\leq \frac{C\varepsilon^{1/2}}{A_1};
		\]
		
		\item[(3)]  $ A_1= A_2\neq 0$. In this case,
		\[
		|B_i^2-A_i^2|<C\varepsilon^{1/2}\quad\Longrightarrow\quad |B_i-A_i|<\frac{C\varepsilon^{1/2} }{A_i},
		\]
		where $i=1,2.$
		Therefore,
		\[
		\mathrm{dist}_2\bigl(v_t,\mathbf O\bigr)
		< \frac{C\varepsilon^{1/2} }{A_1}.
		\]
	\end{itemize}
	The proof is complete.
\end{proof}

	\bigskip
	\noindent{\bf Acknowledgements:}
	G. Wang was supported by National Natural Science Foundation of China
	(Grant No. 12471101)
	and Fundamental Research Funds for the Central Universities
	(Grant No. DUT23RC(3)077).

	\bigskip
	\noindent{\bf  Data Availability} Data sharing not applicable to this article as no datasets were generated or analysed during the current study.

	\bigskip
	\noindent{\bf Conflict of interest}    The authors declare that they have no conflict of interest to this work.
	
	\phantom{s}
	\thispagestyle{empty}

\end{document}